\documentclass[11 pt]{article}
\usepackage[utf8]{inputenc}
\usepackage{amsmath}
\usepackage{amsfonts}
\usepackage{amssymb}
\usepackage{graphicx}
\usepackage{mathrsfs}
\usepackage{upref,amsthm,amsxtra,exscale}
\usepackage{cite}
\usepackage[colorlinks=true,urlcolor=blue,
citecolor=red,linkcolor=blue,linktocpage,pdfpagelabels,
bookmarksnumbered,bookmarksopen]{hyperref}
\usepackage{cleveref}
\usepackage{fullpage}

\usepackage{subcaption}
\usepackage{caption}

\newtheorem{theorem}{Theorem}[section]
\newtheorem{corollary}[theorem]{Corollary}
\newtheorem{remark}[theorem]{Remark}
\newtheorem{lemma}[theorem]{Lemma}
\newtheorem{proposition}[theorem]{Proposition}

\numberwithin{equation}{section}

\newcommand{\R}{\mathbb R}
\newcommand{\N}{\mathbb N}
\newcommand{\cH}{{\mathcal H}}
 
 \def\eps{\varepsilon}

 \usepackage{marvosym}
 
\usepackage[dvipsnames]{xcolor}

\title{Small order limit of fractional Dirichlet sublinear-type problems}

\author{Felipe Angeles\footnote{
Instituto de Matemáticas, Universidad Nacional Autónoma de México, Circuito Exterior, Ciudad Universitaria, 04510 Coyoacán, Ciudad de México, Mexico, \texttt{teojkd@ciencias.unam.mx} 
}\ \ \& Alberto Saldaña\footnote{(\Letter \ Corresponding author) 
Instituto de Matemáticas, Universidad Nacional Autónoma de México, Circuito Exterior, Ciudad Universitaria, 04510 Coyoacán, Ciudad de México, Mexico, \texttt{alberto.saldana@im.unam.mx}}}

\date{}

\begin{document}

\maketitle

\begin{abstract}
    
    We study the asymptotic behavior of solutions to various Dirichlet sublinear-type problems involving the fractional Laplacian when the fractional parameter $s$ tends to zero. Depending on the type on nonlinearity, positive solutions may converge to a characteristic function or to a positive solution of a limit nonlinear problem in terms of the logarithmic Laplacian, that is, the
pseudodifferential operator with Fourier symbol $\ln(|\xi|^2)$. In the case of a logistic-type nonlinearity, our results have the following biological interpretation: in the presence of a toxic boundary, species with reduced mobility have a lower saturation threshold, higher survival rate, and are more homogeneously distributed. As a result of independent interest, we show that sublinear logarithmic problems have a unique least-energy solution, which is bounded and Dini continuous with a log-Hölder modulus of continuity.

\medbreak


\end{abstract}

\section{Introduction}

Consider a positive solution of a sublinear-type problem such as 
\begin{align*}
    (-\Delta)^{s} u_s = f(u_s)\quad \text{ in }\Omega,\qquad u_s=0\quad \text{ on }\R^N\backslash\Omega,
\end{align*} 
where $s\in(0,1)$, $N\geq 1$, $\Omega\subset \R^N$ is an open bounded Lipschitz set, and $f(u)$ is a sublinear-type nonlinearity such as $f(u)=|u|^{p-2}u$ with $p\in(1,2)$ or a bistable nonlinearity such as $f(u)=ku-|u|^{q-1}u$ with $k>0$ and $q>1$.  Here, $(-\Delta)^{s}$ is the fractional Laplacian of order $2s$ given by
\begin{align*}
(-\Delta)^{s}u(x):=c_{N,s} \emph{p.v.} \int_{\mathbb{R}^{N}}\frac{u(x)-u(y)}{|x-y|^{N+2s}}\, dy,\quad 
c_{N,s}:=s(1-s)\frac{\Gamma(\tfrac{N}{2}+s)4^{s}}{\Gamma(2-s)\pi^{\frac{N}{2}}},    
\end{align*}
and \emph{p.v} stands for the integral in the principal value sense. 

In this paper, we study the asymptotic profile of positive solutions $u_s$ as $s\to 0^+$. This asymptotic analysis has only been done for superlinear problems in \cite{HSS22} for least energy solutions and for linear problems in \cite{CW19,FJW20}. The motivation behind the understanding of these profiles is twofold.  On one hand, the parameter $s$ plays an important role in some models coming from population dynamics \cite{CDV17,PV18}, optimal control \cite{SV17}, approximation of fractional harmonic maps \cite{ABS21}, and fractional image denoising \cite{AB17}. In these models, a small value for the fractional parameter $s$ can yield an optimal choice; for instance, for the population models in \cite{CDV17,PV18}, it can happen that a species survives only for dispersal strategies associated to a small value of $s$ (for more information and references we refer to \cite{HSS22}). Another motivation comes from the understanding of the interesting underlying mathematical structures behind the asymptotic profiles of weak solutions as $s\to 0$.  Indeed, in this paper we show that sublinear and superlinear problems have very different behaviors as $s\to 0^+$ and the challenges to characterize the limits are also distinct. 

We begin by discussing the paradigmatic case of the power nonlinearity.  Let $(s_n)_{n\in\N}\subset (0,1)$ and $(p_{n})_{n\in\N}\subset (1,2)$ be such that $\lim\limits_{n\to \infty}s_n=0$ and $\lim\limits_{n\to \infty}p_n=p\in[1,2]$ and consider the equation 
\begin{align}\label{eq:peq}
    (-\Delta)^{s_n} u_n = |u_n|^{p_{n}-2}u_{n}\quad \text{ in }\Omega,\qquad u_{n}=0\quad \text{ on }\R^N\backslash \Omega.
\end{align} 
Since $p_n\in(1,2)$, the problem \eqref{eq:peq} has a unique positive solution for every $n\in\N$ (see, for instance, \cite[Section 6]{BFMST18}), which can be found by global minimization of an associated energy functional (see Section \ref{notation}). Furthermore, these solutions are uniformly bounded independently of $n$, see Proposition~\ref{prop1} below.  This is one of the advantages of the sublinear regime, since similar uniform bounds for superlinear powers in the small order limit are not known. 

Heuristically, it is easy to see that the asymptotic behavior of the sequence of positive solutions $(u_{n})_{n\in\N}$ is closely related to the limit $p$ of the sequence $(p_n)_{n\in \N}$. Indeed, if $p\in [1,2)$, we are led (at least formally) to the limit equation 
\begin{align}\label{eqpm2}
u=u^{p-1}\quad \text{ in }\Omega,    
\end{align}
 where we have used that $(-\Delta)^s$ goes in some suitable sense to the identity operator as $s\to 0^+$ (see, \emph{e.g.}, \cite[Proposition 4.4]{DNPV12}).  This suggests that the limiting profile of the sequence $(u_n)_{n\in\N}$ must be (piecewisely) constant.  On the other hand, if $p=2$, then the limit equation becomes the trivial identity $u=u$, which does not provide information on the asymptotic profile. In this case, similarly as in \cite{HSS22}, we need to consider a first order expansion in $s$ of the fractional Laplacian $(-\Delta)^s$.

As a consequence of the discussion above, we split our analysis of \eqref{eq:peq} in two cases depending on the limit $p$ of the sequence $p_n$. The following result focuses on the case $p=2$.
\begin{theorem}\label{case2}
Let $(s_n)_{n\in\N}\subset (0,1)$ and $(p_n)_{n\in\N}\subset(1,2)$ be such that
\begin{align}\label{snpn}
\lim\limits_{n\to\infty}s_n=0,\quad \lim\limits_{n\to\infty}p_n=2,\quad \text{ and }\quad \mu:=\lim_{n\to\infty}\frac{2-p_n}{s_n}\in(0,\infty).     
\end{align}
Let $u_{n}$ be a positive solution of \eqref{eq:peq}, then $u_{n}\rightarrow u_{0}$ in $L^{q}(\mathbb{R}^{N})$ as $n\rightarrow\infty$ for all $1\leq q<\infty,$ where $u_0\in\mathbb{H}(\Omega)\cap L^{\infty}(\Omega)\backslash \{0\}$ is the unique nonnegative least energy solution of 
\begin{align}
L_{\Delta}u_{0}=-\mu \ln(|u_{0}|)u_{0}\quad\text{ in }\Omega,\qquad u_{0}=0\quad\text{ on }\mathbb{R}^{N}\setminus\Omega.\label{eq:727}
\end{align}
\end{theorem}

Here $L_\Delta$ stands for the logarithmic Laplacian, whose weak solutions belong to a suitable Hilbert space $\mathbb {H}(\Omega)$ (see \eqref{Hdef} below).  The logarithmic Laplacian appears naturally as the first order expansion of the fractional Laplacian; in particular,
\begin{align}
\lim_{s\rightarrow 0^{+}}\left\lvert\frac{(-\Delta)^{s}\varphi-\varphi}{s}-L_{\Delta}\varphi\right\rvert_{p}=0\qquad\mbox{for all}~1<p\leq\infty \text{ and }\varphi\in C^\infty_c(\R^N),\label{eq:726}
\end{align}
where $|\cdot|_p$ denotes the usual $L^p$-norm, see \cite[Theorem 1.1]{CW19}. These type of operators are also related to geometric stable Lévy processes, we refer to \cite{bass94,beghin14,feulefack2021nonlocal,feulefack2021logarithmic,FKT20,MR3626549,vsikic06,JSW20,LW21,CV22} and the references therein for an overview of the different applications that they have (in engineering, finances, physics, mathematics, etc).  For precise definitions and further properties of the logarithmic Laplacian and of the Hilbert space $\mathbb {H}(\Omega)$, we refer to Section \ref{notation} below. We also refer to Remark \ref{rmk:r:1} for a version of Theorem \ref{case2} without sequences (see also Remark \ref{rmk:r:2}).

As a byproduct of Theorem~\ref{case2}, we obtain the following qualitative information on the unique (up to a sign) least energy solution of the limit logarithmic problem.
\begin{theorem}\label{exisleast:intro}
For every $\mu>0$ there is a unique (up to a sign) least energy solution of 
\begin{align}\label{s}
L_{\Delta}v=-\mu\ln(|v|)v\quad\mbox{in}~\Omega,\qquad v=0\quad \text{ on }\R^N\backslash \Omega,
\end{align}
which is a global minimizer of the energy functional
\begin{align}\label{J0def}
J_{0}:\mathbb{H}(\Omega)\rightarrow\mathbb{R},\quad J_{0}(u):=\frac{1}{2}\mathcal{E}_{L}(u,u)+I(u),\quad I(u):=\frac{\mu}{4}\int_{\Omega}u^{2}\left(\ln(u^{2})-1\right)\, dx.
\end{align}
Moreover, $v$ does not change sign and
\begin{align}\label{C0clthm}
0<\sup_{x\in\Omega}|v(x)|\leq (R^{2}e^{\frac{1}{2}-\rho_{N}})^{\frac{1}{\mu}},\qquad \text{where $R:=2\operatorname{diam}(\Omega)$}
\end{align}
and $\rho_N$ is an explicit constant given in \eqref{rhon}. Furthermore, if $\Omega$ satisfies a uniform exterior sphere condition, then $|v|>0$ in $\Omega$, $v\in C(\R^N)$, and there are $\alpha\in (0,1)$ and $C>0$ such that
\begin{align}\label{lhr}
 \sup_{\substack{x,y\in\R^N \\ x\neq y}}\frac{|v(x)-v(y)|}{\ell^\alpha(|x-y|)}<C,\qquad \ell(r):=\frac{1}{|\ln(\min\{r,\tfrac{1}{10}\})|}.
\end{align}
\end{theorem}

Theorems \ref{case2} and \ref{exisleast:intro} are the sublinear counterparts of \cite[Theorem 1.1]{HSS22} and \cite[Theorem 1.2]{HSS22}.  A crucial difference between these results is the sign of $\frac{p_n-2}{s_n}$, which is positive for superlinear problems and negative in the sublinear regime.  This means that, for logarithmic problems, a notion of sublinearity is encoded in the negative sign in front of the coefficient $\mu$ in \eqref{s}.  This sign has several consequences on the asymptotic analysis and on the qualitative properties of the limiting profile.  One key feature in the sublinear case is that the sequence of positive solutions of \eqref{eq:peq} is uniformly bounded (see Proposition~\ref{prop1}). This boundedness is then inherited to the limiting profile, which is the first step to characterize further regularity properties (observe that \eqref{lhr} is a lower-order log-Hölder estimate, see Remark~\ref{reg:rmk}). Here the asymptotic analysis done in Theorem~\ref{case2} is essential, since it is not clear how to obtain a bound as in \eqref{C0clthm} directly from the equation \eqref{s}. Another important difference is the uniqueness of positive solutions, which does not hold in general for superlinear fractional problems (see, for example, \cite[Theorem 1.2]{DRS17} or \cite[Remark 2,11]{DIS22} for a multiplicity result). An $L^\infty$-bound and the uniqueness properties of solutions are not known for logarithmic problems in the \textquotedblleft superlinear regime" ($\mu<0$), see \cite{HSS22}.

Furthermore, methodologically, the treatment of sublinear problems requires a different approach with respect to its superlinear counterpart; for example, \cite[Theorems 1.1 and 1.2]{HSS22} are strongly based on Sobolev logarithmic inequalities; but these do not play any role in our asymptotic analysis.  Instead, we use Fourier transforms, sharp regularity bounds, and direct integral estimates to find a uniform bound of the solutions of \eqref{eq:peq} in the norm of $\mathbb H(\Omega)$ (see Theorem~\ref{boundH}).  This bound together with the compact embedding $\mathbb H(\Omega) \hookrightarrow L^2(\Omega)$ gives the main compactness argument to characterize the limiting profile. We also mention that the uniqueness property stated in Theorem \ref{exisleast:intro} relies strongly on the fact that $\mu>0$ (see the proof of Theorem \ref{exisleast}).  If $\mu<0$, then uniqueness or multiplicity results for \eqref{s} are not known. 

These arguments, however, cannot be used if the limit of the sequence of powers $p_n$ is strictly less than $2$, because in that case the logarithmic Laplacian does not relate in any way to the limit equation \eqref{eqpm2}. Our next result summarizes our asymptotic analysis for \eqref{eq:peq} when $p\in[1,2)$.
\begin{theorem}\label{main:power}
Let $(s_n)_{n\in\N}\subset (0,1)$ and $(p_{n})_{n\in\N}\subset (1,2)$ be such that $\lim\limits_{n\to \infty}s_n=0$ and $\lim\limits_{n\to \infty}p_n=p\in[1,2)$, and let $u_n$ be the unique positive solution of \eqref{eq:peq}.  Then,
\begin{align*}
u_{n}\rightarrow 1\qquad \text{in $L^{q}(\Omega)$ as $n\rightarrow\infty$ for any $1\leq q<\infty$.}
\end{align*}
\end{theorem}

The main difficulty in showing Theorem~\ref{main:power} comes from the absolute lack of compactness tools.  Indeed, as $n\to \infty$, the Sobolev norm $\|\cdot\|_{s_n}$ converges to the $L^2-$norm $|\cdot|_2$ (see, \emph{e.g.}, \cite[Corollary 3]{BBM01}), and therefore it is not possible to use any type of Sobolev embedding. Similarly, all Hölder regularity estimates for $u_n$ degenerate in the limit $s\to 0^+$. Furthermore, since the logarithmic Laplacian does not relate to the limit equation \eqref{eqpm2}, the compactness properties of the space $\mathbb H(\Omega)$ cannot be used. However, since, heuristically, the limit equation is given by \eqref{eqpm2}, it is easy to guess that the limiting profile must be the characteristic function of the set $\Omega$.  As a consequence, this asymptotic analysis is the opposite of that of Theorem~\ref{case2}, since we \textquotedblleft know" a priori the limiting profile, but we do not have any compact embedding at our disposal.  This requires a new approach.

To show Theorem~\ref{main:power}, we use an auxiliary nonlinear eigenvalue problem. To be more precise, consider
\begin{align*}
    \Lambda_n:=\inf\{\|v\|^2_{s_n}\::\: v\in \cH^s_0(\Omega),\ |v|_{p_n}=1\},
\end{align*}
where $\cH^s_0(\Omega)$ is the homogeneous fractional Sobolev space given by 
\begin{align*}
\cH^{s}_{0}(\Omega):=\left\lbrace u \in H^{s}(\mathbb{R}^{N}):u=0~\mbox{on}~\mathbb{R}^{N}\setminus\Omega\right\rbrace
\end{align*}
and 
\begin{align}\label{sf}
\|u\|_{s_n}:=\left(c_{N,s_n}\int_{\R^N}\int_{\R^N}\frac{|u(x)-u(y)|^2}{|x-y|^{N+2s_n}}\, dx\, dy\right)^\frac{1}{2},\qquad |u|_{p_n}:=\left(\int_{\R^N}|u|^{p_n}\, dx\right)^\frac{1}{p_n}.
\end{align}

 A minimizer of $\Lambda_n$ is (after a suitable rescaling) a solution of \eqref{eq:peq}, but the $L^{p_n}$-normalization will turn out to be a useful tool in the asymptotic analysis. Indeed, we show that $(\Lambda_{n})_{n\in \N}$ converges to $\Lambda_{0}>0$ given by 
 \begin{align*}
    \Lambda_0:=\inf\left\{\int_{\Omega}|v|^{2}\, dx\::\: v\in L^2(\Omega),\ \int_{\Omega}|v|^{p}\, dx=1\right\}>0.
\end{align*}
Note that this variational problem does not have any kind of differential operator and a minimizer is achieved at a characteristic function of $\Omega$ (see Lemma~\ref{lem:lam0}). From this fact, we derive that the minimizers $v_n$ of $\Lambda_{n}$ converge to 1 in $L^{2}(\Omega)$. Finally, we use that the solutions $u_n$ of \eqref{eq:peq} are related to $v_n$ by a direct rescaling to obtain the convergence of $u_n$.

Theorems \ref{case2} and \ref{main:power} show that sublinear problems behave very differently than their superlinear counterparts. Moreover, a link between the cases $p<2$ and $p=2$ resides in the assumption $\mu\in(0,\infty)$ required in Theorem~\ref{case2}.  If $\mu=0$, then the limit problem cannot be characterized by the logarithmic Laplacian. To analyze this case, it would be necessary to consider a second (or higher) order expansion of the fractional Laplacian in the parameter $s$.

In the last result we present here, we show that, with some adjustments, a similar strategy can also be used to characterize the limiting profile of other sublinear-type fractional problems.  For instance, consider the nonlinearity $f(u)=ku-u^{p}$ for $k>1$, $p>1$, and $u\geq0$. This nonlinearity is widely studied in the literature; in particular, $p=2$ (the logistic nonlinearity) is used in ecology in the study of population dynamics, where $k$ is a birth rate and $-u^2$ is called a concentration or saturation term (see, \emph{e.g.}, \cite{CDV17,PV18} and the references therein); and $p=3$ (the Allen-Cahn nonlinearity) is used in the study of phase transitions in material sciences (see, \emph{e.g.}, \cite{MSW19} and the references therein).  In this regard, we have the following.

\begin{theorem}\label{thm:f}
Let $k>1$ and $p>1$. There is $s_0=s_0(\Omega,k)\in(0,1)$ so that, for $s\in(0,s_0)$, there is a unique positive solution $u_s\in \cH^{s}_0(\Omega)\cap L^{p+1}(\Omega)$ of
\begin{align}\label{f:p}
(-\Delta)^{s}u_{s}=k u_s - u_s^{p}\text{ in }\Omega, \qquad u_s=0\quad \text{ in }\R^N\backslash\Omega.
\end{align}
Moreover, $u_{s}\to (k-1)^\frac{1}{p-1}$ in $L^{q}(\Omega)$  as $s\rightarrow 0^+$ for every $1\leq q<\infty$.
\end{theorem}

This result has an interesting biological interpretation in terms of population dynamics (at equilibrium): \emph{in the presence of a toxic boundary, species with limited mobility have a lower saturation threshold, higher survival rate, and are more homogeneously distributed}. Indeed, to fix ideas consider $p=2$, $k=2$, let $u_n$ represent the population density of a species, $\Omega=B_R(0)$ be a ball of radius $R>0$, and let $s$ be a parameter describing a diffusion strategy.  Because the nonlinearity $2u-u^2$ has a concentration term, the population density $u_s$ is bounded by 2 (see Proposition~\ref{infbound}). This bound is optimal, in the sense that $u_s$ has values arbitrarily close to 2 as $R\to\infty$ (a heuristic way to see this, is to consider the rescaled equation $R^{-2s}(-\Delta)^{s} v_s = 2v_s-v_s^2$ in $B_1(0)$, with $v_s(x)=u_s(Rx)$, then, letting $R\to \infty$ yields the limit equation $0=2v-v^2$ which implies $v=2$). However, Theorem~\ref{thm:f} yields that $u_s\to 1$ as $s\to 0^+$, independently of $R>0$. This shows that $u_s$ grows only half as much as more dynamical species in large domains for $s$ sufficiently small.  On the other hand, the Dirichlet boundary conditions represent a toxic boundary, which in small domains can be deadly for the species; in fact, for every $s\in(0,1)$ fixed, there is $R>0$ small such that the only solution of \eqref{f:p} is $u\equiv 0$. But again, Theorem~\ref{thm:f} shows that almost static populations thrive even in small domains.  This is consistent with the results and interpretations from \cite{CDV17,PV18}.

Theorem~\ref{thm:f} is a particular case of a slightly more general result, Theorem~\ref{ac1convergence} in Section \ref{gen:sec}.  The proof of Theorem~\ref{thm:f} follows a similar strategy as in Theorem~\ref{main:power}, we begin by considering a nonlinear eigenvalue problem given by
\begin{align}\label{aux:p}
\Theta_s:=\inf\left\{\frac{\|u\|_{s}^2}{2}+\frac{|u|^{p+1}_{p+1}}{p+1}\::\: u\in\mathcal{H}_{0}^{s}(\Omega)\cap L^{p+1}(\Omega)\quad\text{ and }\quad \frac{\varepsilon|u|_2^2}{|\Omega|}=1\right\},
\end{align}
where $\eps>0$ is a parameter. We show that $\Theta_s\to \Theta_0$ as $s\to 0^+$, where
\begin{align*}
\Theta_0:=\inf\left\{\frac{|u|_{2}^2}{2}+\frac{|u|^{p+1}_{p+1}}{p+1}\::\: u\in L^2(\Omega)\cap L^{p+1}(\Omega),\quad u=0\text{ in }\R^N\backslash \Omega,\quad\text{ and }\quad \frac{\varepsilon|u|_2^2}{|\Omega|}=1\right\},
\end{align*}
which is shown to be achieved at $u_0=\eps^{-\frac{1}{2}}\chi_\Omega$.  Note that, in these cases, the functionals have terms with different homogeneities and therefore the link between a minimizer of \eqref{aux:p} and a solution of \eqref{f:p} cannot be established by a direct rescaling.  Here is where the parameter $\eps>0$ is used. A suitable choice of this parameter allows us to link, via a stability-type argument (see \eqref{ss}), the problems \eqref{aux:p} and \eqref{f:p}, and to conclude the desired convergence.

To close this introduction, we mention that an interesting problem would be to consider also \emph{sign-changing} solutions of \eqref{f:p} and to characterize its limit as $s\to 0^+$.  In this case, there is no clear candidate for the limiting profile, and a deeper understanding of the asymptotic behavior of the nodal set is needed (one can compare this analysis with the results from \cite{MSW19}). It could also be interesting to consider other nonlinearities, for instance $f_1(u)=u(u-\alpha)(\beta-u)$, where $\beta>\alpha>0$, or $f_2(u)=\lambda u^q+u^{2^*_s-1}$, where $q\in(0,1)$ and $2^*_s$ is the fractional Sobolev critical exponent.  The nonlinearity $f_1$ is related to the \emph{Allee effect} and it is used in ecology and genetics to establish a correlation between population size and the mean individual fitness \cite{cantrell2004spatial}, whereas $f_2$ is a \emph{concave-convex} nonlinearity for which multiplicity of positive solutions is known in fractional problems \cite{barrios2015critical}. In these cases, formally, the limit equation ($u=f_i(u)$) would have two positive constant solutions. We expect that ground states converge to the least-energy constant with respect to a limit energy functional. 

The paper is organized as follows. In Section \ref{notation} we fix some notation that is used throughout the paper. Section \ref{aux:lem} contains some auxiliary estimates. Section \ref{sec:pnl} is devoted to the power nonlinearity case and it contains the proofs of Theorems \ref{case2}, \ref{exisleast:intro}, and \ref{main:power}.  Finally, in Section \ref{gen:sec} we show Theorem~\ref{ac1convergence}, which directly implies Theorem~\ref{thm:f}.

\section{Notation}\label{notation}

We fix some notation that is used throughout the paper. The space $\cH^s_0(\Omega)$ is the homogeneous fractional Sobolev space given by 
\begin{align*}
\cH^{s}_{0}(\Omega):=\left\lbrace u \in H^{s}(\mathbb{R}^{N}):u=0~\mbox{on}~\mathbb{R}^{N}\setminus\Omega\right\rbrace.
\end{align*}
The energy functional associated to \eqref{eq:peq} is $J_{s_n}:\cH^{s_n}_0(\Omega)\to \R$ given by
\begin{align}\label{Jdef}
J_{s_n}(u):=\frac{1}{2}\|u\|^{2}_{s_n}-\frac{1}{p_{n}}|u|_{p_{n}}^{p_{n}},
\end{align}
where $\|u\|_{s_n}$ and $|u|_{p_n}$ are norms defined in \eqref{sf}.  We also let $|u|_\infty$ denote the usual supremum norm. Following \cite{CW19}, the logarithmic Laplacian $L_{\Delta}$ can be evaluated as
\begin{align*}
L_{\Delta}u(x):=c_{N}~p.v.~\int_{B_{1}(x)}\frac{u(x)-u(y)}{|x-y|^{N}}\, dy-c_{N}\int_{\mathbb{R}^{N}\setminus B_{1}(x)}\frac{u(y)}{|x-y|^{N}}\, dy+\rho_{N}u(x),
\end{align*}
where 
\begin{align}\label{rhon}
c_{N}:=\pi^{-\tfrac{N}{2}}\Gamma(\tfrac{N}{2}),\qquad \rho_{N}:=2\ln 2+\psi(\tfrac{N}{2})-\gamma,\quad \text{ and }\quad\gamma:=-\Gamma^{\prime}(1).
\end{align}
Here $\gamma$ is also known as the Euler-Mascheroni constant and $\psi:=\frac{\Gamma^{\prime}}{\Gamma}$ is the digamma function.  Moreover, $\mathbb{H}(\Omega)$ is the Hilbert space given by
\begin{align}\label{Hdef}
\mathbb{H}(\Omega):=\left\lbrace u\in L^{2}(\mathbb{R}^{N})~:~\int\int_{\substack{x,y\in\mathbb{R}^{N}\\ |x-y|\leq 1}}\frac{|u(x)-u(y)|^{2}}{|x-y|^{N}}\, dx\, dy<\infty~\mbox{and}~u=0~\mbox{in}~\mathbb{R}^{N}\setminus\Omega\right\rbrace
\end{align}
with inner product 
\begin{align*}
\mathcal{E}(u,v):=\frac{c_{N}}{2}\int\int_{\substack{x,y\in\mathbb{R}^{N}\\ |x-y|\leq 1}}\frac{(u(x)-u(y))(v(x)-v(y))}{|x-y|^{N}}\, dx\, dy,
\end{align*}
and the norm $\|u\|:=\left(\mathcal{E}(u,u)\right)^{\tfrac{1}{2}}$.  The space of compactly supported smooth functions $C^\infty_c(\Omega)$ is dense in $\mathbb{H}(\Omega)$, see \cite[Theorem 3.1]{CW19}. The operator $L_{\Delta}$ has the following associated quadratic form 
\begin{align}
\mathcal{E}_{L}(u,v)
&:=\mathcal{E}(u,v)-c_{N}\int\int_{\substack{x,y\in\mathbb{R}^{N}\\ |x-y|\geq 1}}\frac{u(x)v(y)}{|x-y|^{N}}\, dx\, dy+\rho_{N}\int_{\mathbb{R}^{N}}uv\, dx
\label{eq:bilog}.
\end{align}
Furthermore, for $u\in\mathbb H(\Omega)$,
\begin{align}\label{homega}
\mathcal{E}_{L}(u,u)=\frac{c_{N}}{2}\int_\Omega\int_\Omega \frac{(u(x)-u(y))^2}{|x-y|^{N}}\, dx\, dy
+\int_\Omega (h_\Omega(x)+\rho_N)u(x)^2\, dx,
\end{align}
where $h_\Omega(x)=c_N(\int_{B_1(x)\backslash \Omega}|x-y|^{-N}\, dy-\int_{ \Omega\backslash B_1(x)}|x-y|^{-N}\, dy)$, see \cite[Proposition 3.2]{CW19}.

By \cite[Theorem 1.1]{CW19}, it holds that 
\begin{align}
\mathcal{E}_{L}(u,u)=\int_{\mathbb{R}^{N}}\ln(|\xi|^2)|\hat{u}(\xi)|^{2}\, d\xi\qquad\mbox{for all}~u\in\mathcal{C}_{c}^{\infty}(\Omega),\label{eq:bilogf}
\end{align}
where $\hat u$ is the Fourier transform of $u$. Moreover, for $\varphi\in\mathcal{C}_{c}^{\infty}(\Omega)$ we have that $L_{\Delta}\varphi\in L^{p}(\mathbb{R}^{N})$ and 
\begin{align}\label{eq:725}
\mathcal{E}_{L}(u,\varphi)=\int_{\Omega}uL_{\Delta}\varphi \, dx
 \qquad \text{ for $u\in\mathbb{H}(\Omega)$,}
\end{align}
see \cite[Theorem 1.1]{CW19}. We say that $u\in\mathbb{H}(\Omega)$ is a weak solution of \eqref{eq:727} if 
\begin{align}
\mathcal{E}_{L}(u,v)=-\mu\int_{\Omega}uv\ln|u|\, dx\quad\mbox{for all}~v\in\mathbb{H}(\Omega).
\end{align}
Note that $\lim_{t\to 0}\ln(t^2)t =0$.

\section{Auxiliary lemmas}\label{aux:lem}

\subsection{Asymptotic estimates}
\begin{lemma}\label{lem2.2}
Let $(\varphi_{n})_{n\in \N}$ be a uniformly bounded sequence in $L^{\infty}(\Omega)$, $p\in [1,2]$, and let $(p_n)_{n\in \N}\subset (1,2)$ be such that $\lim_{n\to\infty}p_n=p$. Then, 
\begin{align}
    \int_{\Omega}||\varphi_{n}|^{p_{n}}-|\varphi_{n}|^{p}|\, dx\rightarrow 0\quad\mbox{as}\quad n\rightarrow\infty.
\end{align}
\end{lemma}
\begin{proof}
Consider the function $g(t):=|\varphi_{n}|^{t}$. Then,
\begin{align*}
|\varphi_{n}|^{p_{n}}-|\varphi_{n}|^{p}
=\int_{0}^{1}g^{\prime}(p+\tau(p_{n}-p))(p_{n}-p)\, d\tau=\int_{0}^{1}\ln\left(|\varphi_{n}|\right)|\varphi_{n}|^{p+\tau(p_{n}-p)}(p_{n}-p)\, d\tau.
\end{align*}
Integrating in $\Omega$ and using Fubini's Theorem,
\begin{equation}
\int_{\Omega}\left\lvert|\varphi_{n}|^{p_{n}}-|\varphi_{n}|^{p}\right\rvert \, dx\leq\int_{0}^{1}\int_{\Omega}\left\lvert\ln\left(|\varphi_{n}|\right)\right\rvert|\varphi_{n}|^{p+\tau(p_{n}-p)}|p_{n}-p|\, dx\, d\tau.
\label{eq:25}
\end{equation}
By assumption, there is $M>2$ such that $|\varphi_{n}|_{\infty}\leq M$ for all $n\in\mathbb N$.  Therefore, by \eqref{eq:25},
\begin{align*}
    \int_{\Omega}\left\lvert|\varphi_{n}|^{p_{n}}-|\varphi_{n}|^{p}\right\rvert \, dx\leq
    |p_n-p| |\ln M| M^{p+1}|\Omega|
\end{align*}
for all $n$ sufficiently large, and the claim follows. 
\end{proof}

\begin{lemma}\label{d:lemma}
Let $(s_k)_{k\in\N}\subset(0,\frac{1}{4})$ and $(p_k)_{k\in\N}\subset(1,2)$ be such that $\lim\limits_{k\to\infty}s_k=0$ and $\lim\limits_{k\to\infty}p_k=2$. Let
$(u_{k})_{k\in\mathbb N}\subset L^{2}(\Omega)$ and $u_{0}\in L^{2}(\Omega)$ be such that $u_{k}\rightarrow u_{0}$ in $L^{2}(\Omega)$ as $k\rightarrow\infty$. Then, passing to a subsequence, 
\begin{align*}
\lim_{k\rightarrow\infty}\int_{\Omega}\ln(|u_{k}|^{2})|u_{k}|^{p_{k}-2}u_{k}\varphi \, dx&=\int_{\Omega}\ln(|u_{0}|^{2})u_{0}\varphi \, dx\qquad \text{for all $\varphi\in\mathcal{C}_{c}^{\infty}(\Omega).$}
\end{align*} 
\end{lemma}
\begin{proof}
Notice that
\begin{align}\label{eqm2}
&\int_{\Omega}\ln(|u_{k}|^{2})|u_{k}|^{p_{k}-2}u_{k}\varphi\, dx=\int_{\left\lbrace|u_{k}|\leq 1\right\rbrace}\ln(|u_{k}|^{2})|u_{k}|^{p_{k}-2}u_{k}\varphi\, dx+\int_{\left\lbrace|u_{k}|>1\right\rbrace}\ln(|u_{k}|^{2})|u_{k}|^{p_{k}-2}u_{k}\varphi \, dx.
\end{align}
Passing to a subsequence, we have that $\sup_{t\in(0,1)}t^{p_{k}-1}|\ln t^{2}|\leq \sup_{t\in(0,1)}t^{\frac{1}{2}}|\ln t^{2}|<\infty$ (note that $\ln(t^2)t^\frac{1}{2}=0$) and $u_{k}\rightarrow u_{0}$  a.e. in $\Omega$ as $n\to\infty$. In particular, since $\ln (1)=0$,
\begin{align*}
\chi_{\{|u_k|\leq 1\}}\ln(|u_k|^2)u_{k}\rightarrow \chi_{\{|u_0|\leq 1\}}\ln(|u_0|^2)u_{0}\qquad \text{ a.e. in $\Omega$ as $n\to\infty$.}
\end{align*}
Then, by the dominated convergence theorem,
\begin{align}\label{eqm1}
\lim_{k\to\infty}\int_{\left\lbrace|u_{k}|\leq 1\right\rbrace}\ln(|u_{k}|^{2})|u_{k}|^{p_{k}-2}u_{k}\varphi\, dx
= \int_{\left\lbrace|u_{0}|\leq 1\right\rbrace}\ln(|u_{0}|^{2})u_{0}\varphi\, dx.
\end{align}
If $|u_{k}|>1$, it follows easily (see, for example, \cite[Lemma 3.3]{HSS22} with $\alpha=p_{k}-2$ and $\beta=1$) that, passing to a subsequence,
\begin{align}\label{spiral}
\ln(|u_{k}|^{2})|u_{k}|^{p_{k}-2}|u_{k}\varphi|&\leq \frac{2}{3-p_{k}}|u_{k}|^{2}|\varphi|\leq 2\|\varphi\|_{\infty}|U|^{2}\in L^{1}(\Omega),
\end{align}
for some $U\in L^2(\Omega)$ (see \cite[Lemma A.1]{W97}). The claim now follows by applying the dominated convergence theorem to the second integral in \eqref{eqm2} together with \eqref{eqm1}.
\end{proof}
\begin{lemma}
\label{haprox}
Let $(s_{k})_{k\in\N},$ $(p_k)_{k\in\N},$ and $\mu$ as in \eqref{snpn} and let $\phi\in\mathcal{C}_{c}^{\infty}(\Omega)$. Then,
\begin{align}
\label{eq:weakaprox}
\lim_{k\rightarrow\infty}\frac{1}{s_{k}}J_{s_{k}}(\phi)=-\frac{\mu}{4}|\phi|_{2}^{2}+\frac{1}{2}\left(\mathcal{E}_{L}(\phi,\phi)+\mu\int_{\mathbb{R}^{N}}|\phi|^{2}\ln|\phi|\, dx\right).
\end{align}
In particular, if $v\in\mathbb{H}(\Omega)$ is a weak solution of \eqref{eq:727} and $(\phi_{n})_{n\in\N}\subset\mathcal{C}_{c}^{\infty}(\Omega)$ is such that $\phi_{n}\rightarrow v$ in $\mathbb{H}(\Omega)$ as $n\rightarrow\infty$, then $\lim\limits_{n\rightarrow\infty}\lim\limits_{k\rightarrow\infty}\frac{1}{s_{k}}J_{s_{k}}(\phi_n)=-\frac{\mu}{4}|v|_{2}^{2}.$
\end{lemma}
\begin{proof}
Let $\phi\in\mathcal{C}_{c}^{\infty}(\Omega)$, then, 
\begin{align*}
\lim_{k\rightarrow\infty}\frac{1}{s_{k}}J_{s_{k}}(\phi)&=\lim_{k\rightarrow\infty}\frac{1}{s_{k}}\left(\frac{\|\phi\|_{s_{k}}^{2}}{2}-\frac{|\phi|_{p_{k}}^{p_{k}}}{p_{k}}\right)=\lim_{k\rightarrow\infty}\frac{1}{s_{k}}\left(\frac{1}{2}-\frac{1}{p_{k}}\right)\|\phi\|_{s_{k}}^{2}+\lim_{k\rightarrow\infty}\frac{\|\phi\|_{s_{k}}^{2}-|\phi|_{p_{k}}^{p_{k}}}{p_{k}s_k}.
\end{align*}
Thus, since $\|\phi\|_{s_{k}}^{2}\to |\phi|_{2}^{2}$ (see, \emph{e.g.}, \cite[Corollary 3]{BBM01}),
\begin{align}\label{eq:81}
\lim_{k\rightarrow\infty}\frac{1}{s_{k}}J_{s_{k}}(\phi)=-\frac{\mu}{4}|\phi|_{2}^{2}+\frac{1}{2}\lim_{k\rightarrow\infty}\frac{\|\phi\|_{s_{k}}^{2}-|\phi|_{p_{k}}^{p_{k}}}{s_{k}}.
\end{align}
Note that 
\begin{align}
\label{eq:82}
\frac{\|\phi\|_{s_{k}}^{2}-|\phi|_{p_{k}}^{p_{k}}}{s_{k}}=\mathcal{I}_{k}+\mathcal{J}_{k},\qquad \text{ where }\quad \mathcal{I}_{k}:=\frac{\|\phi\|_{s_{k}}^{2}-|\phi|_{2}^{2}}{s_{k}}\quad\mbox{and}\quad \mathcal{J}_{k}:=\frac{|\phi|_{2}^{2}-|\phi|_{p_{k}}^{p_{k}}}{s_{k}}.
\end{align}

Let $\widehat{\phi}$ denote the Fourier transform of $\phi$.  If $|\xi|<1$, then passing to a subsequence,
\begin{align}
|\xi|^{2s_{k}\tau}|\ln(|\xi|^{2})||\widehat{\phi}(\xi)|^{2}\leq 2|\widehat{\phi}(\xi)|^{2}.\label{eq:83}
\end{align}
On the other hand, if $|\xi|\geq 1$, since $0<s_{k}<\frac{1}{4}$, we have that
\begin{align}
|\xi|^{2s_{k}\tau}\ln(|\xi|^{2})|\widehat{\phi}(\xi)|^{2}\leq |\xi|^{1/2}\ln(|\xi|^{2})|\widehat{\phi}(\xi)|^{2}\leq\frac{4}{3}|\xi|^{2}|\widehat{\phi}(\xi)|^{2}.\label{eq:84}
\end{align}
Then, by \eqref{eq:83}, \eqref{eq:84}, and dominated convergence,
\begin{align}
\lim_{k\rightarrow\infty}\mathcal{I}_{k}
=\lim_{k\rightarrow\infty}\int_{\mathbb{R}^{N}}\int_{0}^{1}|\xi|^{2s_{k}\tau}\ln(|\xi|^{2})|\widehat{\phi}(\xi)|^{2}\, d\tau\, d\xi
=\int_{\mathbb{R}^{N}}\ln(|\xi|^{2})|\widehat{\phi}|^{2}\, d\xi=\mathcal{E}_{L}(\phi,\phi).\label{eq:85}
\end{align}
For $\mathcal{J}_{k}$ it holds that 
\begin{align*}
&\lim_{k\rightarrow\infty}-\mathcal{J}_{k}=\lim_{k\rightarrow\infty}\frac{p_k-2}{s_k}\int_{0}^{1}\int_{\mathbb{R}^{N}}|\phi|^{2+(p_{k}-2)\tau}\ln|\phi|\, dx\, d\tau\\
&=\lim_{k\rightarrow\infty}\frac{p_k-2}{s_k}\int_{0}^{1}\int_{\left\lbrace|\phi|<1\right\rbrace}|\phi|^{2+(p_{k}-2)\tau}\ln|\phi|\, dx\, d\tau
+\lim_{k\rightarrow\infty}\frac{p_k-2}{s_k}\int_{0}^{1}\int_{\left\lbrace|\phi|\geq 1\right\rbrace}|\phi|^{2+(p_{k}-2)\tau}\ln|\phi|\, dx\, d\tau.
\end{align*}
If $|\phi|<1$, $|\phi|^{2+(p_{k}-2)\tau}\ln(|\phi|)$ is bounded independently of $k$. On the other hand, if $|\phi|\geq 1$, $
|\phi|^{2+(p_{k}-2)\tau}\ln(|\phi|)<2|\phi|^{3}\in L^{1}(\mathbb{R}^{N})$ (see \eqref{spiral}). By dominated convergence,
\begin{align}
\lim_{k\rightarrow\infty}\mathcal{J}_{k}=\mu\int_{\mathbb{R}^{N}}|\phi|^{2}\ln|\phi|\, dx.\label{eq:86}
\end{align}
By using \eqref{eq:82}, \eqref{eq:85} and \eqref{eq:86} into \eqref{eq:81} we obtain \eqref{eq:weakaprox}.\\
Now, let $(\phi_{n})_{n\in\N}\subset\mathcal{C}_{c}^{\infty}(\Omega)$ and $v\in\mathbb{H}(\Omega)$ such that $\phi_{n}\rightarrow v$ in $\mathbb{H}(\Omega)$ as $n\to\infty$. Assume that $v\in\mathbb{H}(\Omega)$ is a weak solution of \eqref{eq:727}; in particular, $\mathcal{E}_{L}(v,v)+\mu\int_{\mathbb{R}^{N}}|v|^{2}\ln|v|\, dx=0.$ Since $J_{0}$ is of class $\mathcal{C}^{1}$ over $\mathbb{H}(\Omega)$ (see \cite[Lemma 3.9]{HSS22}), $\mathcal{E}_{L}(\phi_{n},\phi_{n})+\mu\int_{\mathbb{R}^{N}}|\phi_{n}|^{2}\ln|\phi_{n}|\, dx=o(1)$ as $n\rightarrow\infty.$ Then, by the continuous embedding of $\mathbb{H}(\Omega)$ into $L^{2}(\Omega)$, $\frac{\mu}{4}|\phi_{n}|_{2}^{2}\rightarrow\frac{\mu}{4}|v|_{2}^{2}$ as $n\rightarrow\infty.$ This concludes the proof.
\end{proof}

We quote the following result from \cite[Lemma 3.5]{HSS22}.
\begin{lemma}
\label{smallorder}
Let $u\in\mathcal{H}_{0}^{s}(\Omega)$ for some $s\in(0,1)$. Then $u\in\mathbb{H}(\Omega)$ and there is $C_{1}=C_{1}(N)>0$ and $C_2=C_{2}(\Omega)>0$ such that 
$|\mathcal{E}_{L}(u,u)|\leq C_{1}|u|_{1}^{2}+\frac{1}{s}\|u\|_{s}^{2}$ and $\|u\|^{2}\leq C_{2}|u|_{2}^{2}+\frac{1}{s}\|u\|_{s}^{2}.$
\end{lemma}

\subsection{Uniform bounds}

To prove Theorem~\ref{main:power} we need some uniform regularity a priori estimates and a fine analysis of the constants involved.

\begin{lemma}\label{reg:lem}
Let $s\in(0,\frac{1}{4})$, $g\in  L^{N/s^2}(\Omega)$, and let $u$ be a weak solution of $(-\Delta)^s u=g$ in $\Omega$ and $u=0$ in $\mathbb R^N\setminus \Omega.$ Then, 
\begin{align}\label{reg:eq}
\|u\|_{L^\infty(\Omega)}\leq \Big(1+\Big(\ln(R^2)+\tfrac{1}{2}-\rho_{N}\Big)s+o(s)\Big)\|g\|_{L^{N/s^2}(\Omega)}\qquad \text{ as }s\to 0^+,
\end{align}
where $R:=2\operatorname{diam}(\Omega)$ and $\rho_N$ is given in \eqref{rhon}.
\end{lemma}
\begin{proof}
For the first part of the proof, we argue as in \cite[Proposition 1.2]{FRRO16}. We consider the problem
\begin{align}\label{vg}
(-\Delta)^s v=|g|\qquad \text{ in $\R^N$,}
\end{align}
 where $g$ has been extended by zero outside $\Omega$. Using the fundamental solution (see, e.g., \cite[Theorem 5]{Sti19} or \cite[Definition 5.6]{AJS16}), we have the function $v:\R^N\to\R$ given by
\begin{equation}\label{eq:def_v}
v(x)=c_{N,-s}\int_{\Omega}\frac{|g(y)|}{|x-y|^{N-2s}}\, dy,\qquad c_{N,-s}=\frac{\Gamma(\tfrac{N}{2}-s)}{4^s\Gamma(s)\pi^{N/2}},
\end{equation}
is one solution for \eqref{vg} (note that there can be other solutions for \eqref{vg}).  Observe that $v\geq 0$ and, by the comparison principle, $-v\leq u\leq v$, since $-|g|\leq g \leq |g|$. From \eqref{eq:def_v} and H\"older's inequality, we have, for $x\in\Omega$, that
\begin{equation}
0\leq |u(x)|< v(x)=c_{N,-s}\int_{\Omega}\frac{|g(y)|}{|x-y|^{N-2s}}\, dy \leq c_{N,-s} \|g\|_{L^{N/s^2}(\Omega)} \left(\int_{\Omega}|x-y|^{(2s-N)q }d{y}\right)^{1/q},
\end{equation}
where $q=\frac{N}{N-s^2}$.  Without loss of generality, assume that $0\in \Omega$ and let $R:=2\operatorname{diam}(\Omega)>0$. Then $\Omega\subset B_{R/2}(0)$ and, for $x\in \Omega$,
\begin{align*}
\int_{\Omega}|x-y|^{(2s-N)q}\, dy
&\leq\int_{B_R}|y|^{(2s-N)q}\, dy
= |\mathbb S^{N-1}| \int_{0}^{R} \rho^{(2s-N)q}\rho^{N-1}\, d\rho\\
&=\frac{2\pi^{\frac{N}{2}}}{\Gamma(\tfrac{N}{2})}\frac{R^{N(1-q)+2qs}}{N(1-q)+2q s}
=\frac{2\pi^{\frac{N}{2}}}{\Gamma(\tfrac{N}{2})}\frac{R^{t(s)}}{t(s)},
\end{align*}
where $t(s):=N(1-q)+2q s=\frac{N(2-s)s}{N-s^2}$ and $|\mathbb S^{N-1}|=\frac{2\pi^{\frac{N}{2}}}{\Gamma(\frac{N}{2})}.$  Thus, we have proved that $\|u\|_{L^\infty(\Omega)}\leq C_1 \|g\|_{L^{N/s^2}(\Omega)}$, where
\begin{align*}
C_1&=C_1(\Omega,N,s,p)
=\left(\frac{2\pi^{\frac{N}{2}}}{\Gamma(\tfrac{N}{2})}\right)^{\frac{N-s^2}{N}} \frac{\Gamma(\frac{N}{2}-s)}{4^s \Gamma(s)\pi^{\frac{N}{2}}} \left(\frac{R^{t(s)}}{t(s)}\right)^{\frac{N-s^2}{N}}
=:h(s).
\end{align*}
Then, $C_1=h(0)+sh'(0)+o(s)$ as $s\to 0^+$.  A direct calculation shows that $h(0)=\lim_{s\to 0^+}h(s)=1$ and
\begin{align*}
h'(0)
=\lim_{s\to 0^+}h'(s)
=\ln(R^2)+\gamma +\frac{1}{2}-2\ln(2)-\psi\left(\tfrac{N}{2}\right)
=\ln(R^2)+\frac{1}{2}-\rho_{N},
\end{align*}
where $\rho_N$ is given in \eqref{rhon}.  This ends the proof.
\end{proof}

\begin{proposition}\label{prop1}
Let $\Omega\subset\R^N$ be a bounded domain, let $(s_n)_{n\in\N}\subset(0,1)$, $(p_{n})_{n\in \N}\subset (1,2)$ be such that
$\lim_{n\to \infty} s_n =0$, $k:=\lim_{n\to \infty}\frac{s_n}{2-p_n}\in[0,\infty),$ and let $u_n$ be a weak solution of
\begin{equation}\label{eq:nonlinear_ellip}
(-\Delta)^{s_n} u_{n}= |u_n|^{p_{n}-2}u_n \quad \text{in } \Omega, \qquad
u_n=0 \quad \text{in } \R^{N}\setminus \Omega.
\end{equation}
Then $|u_n|_\infty\leq (R^{2}e^{\frac{1}{2}-\rho_{N}})^k+o(1)$ as $n\to\infty,$ where $R:=2\operatorname{diam}(\Omega)$.
\end{proposition}
\begin{proof}
By \cite[Proposition 8.1]{RSV17}, $u_n\in L^\infty(\R^N)$. 
Let $C_1=\ln(R^2)+\frac{1}{2}-\rho_{N}$, where $\rho_N$ is given by \eqref{rhon} and $R:=2\operatorname{diam}(\Omega)>0$. By Lemma~\ref{reg:lem}, for $n$ sufficiently large,
\begin{align*}
|u_n|_\infty &\leq (1+s_nC_1+o(s_n))||u_n|^{p_{n}-1}|_{\frac{N}{s_n^2}} = (1+s_nC_1+o(s_n)) \left(\int_\Omega |u_n|^{\frac{N}{s_n^2}(p_{n}-1)}\ dx \right)^\frac{s_n^2}{N}\\
&\leq (1+s_nC_1+o(s_n))|u_n|_\infty^{p_{n}-1} |\Omega|^\frac{s_n^2}{N}.
\end{align*}
Then, $|u_n|_\infty
\leq  \left((1+s_nC_1+o(s_n))^\frac{1}{s_n}|\Omega|^\frac{s_n}{N}\right)^\frac{s_n}{2-p_n}$.  Let $k=\lim\limits_{n\to \infty}\frac{s_n}{2-p_n} \geq 0$, then 
\begin{align*}
    \lim_{n\to\infty}\left((1+s_nC_1+o(s_n))^\frac{1}{s_n}|\Omega|^\frac{s_n}{N}\right)^\frac{s_n}{2-p_n}
    =e^{kC_1}=(R^{2}e^{\frac{1}{2}-\rho_{N}})^k
\end{align*}
(see \cite[Lemma 3.1]{HSS22}), as claimed. 
\end{proof}

\subsection{Upper and lower energy bounds}

Now we show lower and upper energy bounds for the unique positive solution $u_n$ of \eqref{eq:peq}.  The lower bound is used in the proof of Theorem~\ref{case2}, the upper bound is presented as a result of independent interest and for comparison with the bound given in Proposition \ref{prop1}.

In the following, for each $s\in(0,\tfrac{1}{4})$, $\varphi_{s}$ denotes the first Dirichlet eigenfunction of the fractional Laplacian (normalized in $L^2$-sense) and $\lambda_{1,s}$ its first eigenvalue, that is,
\begin{align}\label{fef}
(-\Delta )^{s}\varphi_{s}=\lambda_{1,s}\varphi_{s}\quad \text{ in $\Omega$},\qquad \varphi_{s}=0\quad \text{ on }\mathbb{R}^{N}\setminus\Omega,\qquad |\varphi_s|_2^2=1.
\end{align}
Due to the variational formulation of the first eigenvalue,
\begin{align}
\label{eq:poincin}
|u|_{2}^{2}\leq\frac{1}{\lambda_{1,s}}\|u\|_{s}^{2}\qquad \text{ for every $u\in\mathcal{H}_{0}^{s}(\Omega)$ and for each $s\in(0,\tfrac{1}{4})$}.
\end{align}

\begin{lemma}
\label{bounds}
Let $(s_n)_{n\in\N}\subset (0,1)$ be such that $\lim\limits_{n\to \infty}s_n=0$, $(p_n)_{n\in \N}\subset (1,2),$ and let $u_{n}$ be a positive solution of \eqref{eq:nonlinear_ellip} then,
\begin{align}
(\lambda_{1,s_n})^{\frac{p_{n}}{p_n-2}}|\Omega|\geq \|u_{n}\|_{s_{n}}^{2}&\geq \lambda_{1,s_n}|\varphi_{s_{n}}|_{2}^{2}\left(\frac{2}{p_{n}}\frac{|\varphi_{s_{n}}|_{p_{n}}^{p_{n}}}{\lambda_{1,s_n}|\varphi_{s_{n}}|_{2}^{2}}\right)^{\frac{2}{2-p_{n}}}
\frac{2^{p_n-2}-1}{p_n-2} \frac{p_n}{2^{p_n}}.\label{eq:71}
\end{align}
\end{lemma}
\begin{proof}
Let $a_{n}:=\lambda_{1,s_n}|\varphi_{s_{n}}|_{2}^{2},$ $b_{n}:=|\varphi_{s_{n}}|_{p_{n}}^{p_{n}},$ $t>0$, and note that
\begin{align*}
J_{s_{n}}(t\varphi_{s_{n}})
=\frac{t^2}{2}\|\varphi_{s_{n}}\|_{s_n}^2-\frac{t^{p_n}}{p_n}|\varphi_{s_{n}}|_{s_n}^{s_n}
=t^2\frac{\lambda_{1,s_n}}{2}|\varphi_{s_{n}}|_{2}^2-\frac{t^{p_n}}{p_n}|\varphi_{s_{n}}|_{s_n}^{s_n}
=t^2\left(\frac{a_n}{2}-t^{p_n-2}\frac{b_n}{p_n}\right).
\end{align*}
Then $J_{s_{n}}(t\varphi_{s_{n}})<0$ if $t<(\frac{2}{p_n}\frac{b_n}{a_n})^\frac{1}{2-p_n}$.
Let $u_n$ be a positive solution of \eqref{eq:nonlinear_ellip}.  Since the least energy solution is the unique positive solution of \eqref{eq:nonlinear_ellip} (see \cite[Section 6]{BFMST18}), we have that $u_n$ is the least energy solution. Let $t_n:=\frac{1}{2}\left(\frac{2}{p_{n}}\frac{b_{n}}{a_{n}}\right)^{\frac{1}{2-p_{n}}}$, then
\begin{align}
\left(\frac{1}{2}-\frac{1}{p_{n}}\right)\|u_n\|_{s_{n}}^{2}=J_{s_{n}}(u_n)\leq J_{s_{n}}(t_n\varphi_{s_{n}})=\frac{a_n}{4}\left(\frac{2}{p_{n}}\frac{b_{n}}{a_{n}}\right)^{\frac{2}{2-p_{n}}}
\left(
\frac{1}{2}-\frac{1}{2^{p_n-1}}
\right)
\label{eq:74}
\end{align}
and the lower bound in \eqref{eq:71} follows.  On the other hand, by \eqref{eq:poincin}, for every $u\in\mathcal{H}_{0}^{s_{n}}(\Omega)$,
\begin{align}
J_{s_{n}}(u)&=\frac{1}{2}\|u\|^{2}_{s_{n}}-\frac{1}{p_{n}}|u|_{p_{n}}^{p_{n}}\geq \frac{1}{2}\|u\|^{2}_{s_{n}}-\frac{1}{p_{n}}C(s_{n},p_{n},\Omega)^{p_{n}}\|u\|_{s_{n}}^{p_{n}},\label{eq:75}
\end{align}
where $C(s_{n},p_{n},\Omega):=(\lambda_{1,s_n})^{-\tfrac{1}{2}}|\Omega|^{\frac{2-p_{n}}{2p_{n}}}$.
For $t\geq 0$ let $f(t):=\frac{1}{2}t^{2}-\frac{1}{p_{n}}C(s_{n},p_{n},\Omega)^{p_{n}}t^{p_{n}}$. Then,  $f^{\prime}(t)=t-C(s_{n},p_{n},\Omega)^{p_{n}}t^{p_{n}-1}=0$ implies that $t_{0}=\left(\frac{1}{C(s_{n},p_{n},\Omega)^{p_{n}}}\right)^{\frac{1}{p_{n}-2}}$ is a critical point of $f$. By computing the second derivative and evaluating we obtain that $f^{\prime\prime}(t_{0})=1-C(s_{n},p_{n},\Omega)^{p_{n}}(p_{n}-1)t_{0}^{p_{n}-2}=2-p_{n}>0,$
implying that $t_{0}$ is the minimizer for $f$. Using $t_{0}$ in $f$ we obtain a lower bound for the energy functional $J_{s_{n}}$, given by $f(t_{0})=\frac{p_{n}-2}{2p_{n}}\left(C(s_{n},p_{n},\Omega)^{p_{n}}\right)^{\frac{2}{2-p_{n}}}.$ Thus, for every $u\in\mathcal{H}_{0}^{s_{n}}(\Omega)$, it holds that
$J_{s_{n}}(u)\geq \frac{p_{n}-2}{2p_{n}}\left(C(s_{n},p_{n},\Omega)^{p_{n}}\right)^{\frac{2}{2-p_{n}}}$.  Therefore,
\begin{align*}
\left(\frac{1}{2}-\frac{1}{p_{n}}\right)\|u_{n}\|_{s_{n}}^{2}=J_{s_n}(u_{n})&\geq \frac{p_{n}-2}{2p_{n}}\left(C(s_{n},p_{n},\Omega)^{p_{n}}\right)^{\frac{2}{2-p_{n}}},
\end{align*}
and the upper bound in \eqref{eq:71} follows.
\end{proof}

Recall that $\varphi_{L}$ denotes the first Dirichlet eigenfunction of the logarithmic Laplacian (normalized in the $L^2$-sense) and $\lambda_{1}^{L}$ its corresponding eigenvalue, that is, $L_{\Delta}\varphi_{L}=\lambda_{1}^{L}\varphi_{L}$ in $\Omega$, $\varphi_{L}=0$ on $\mathbb{R}^{N}\setminus\Omega,$ and $|\varphi_L|_2^2=1$.

\begin{lemma}
\label{limit2}
Let $(s_n)_{n\in\N},$ $(p_n)_{n\in\N},$ and $\mu$ as in \eqref{snpn}, then $
\lim\limits_{n\rightarrow\infty}(\lambda_{1,s_n})^{\frac{p_{n}}{p_n-2}}=\exp\left(-{\frac{2\lambda_{1}^{L}}{\mu}}\right).$
\end{lemma}
\begin{proof}The claim follows from the definition of $\mu$ and the fact that
\begin{align}\label{asymp}
\lambda_{1,s_n}=1+s_n\lambda_{1}^{L}+o(s_n)\qquad \text{ as $n\to\infty$ }
\end{align}
(see \cite[Theorem 1.5]{CW19} or \cite[Theorem 1.1]{FJW20}), because $\lim _{s \rightarrow 0^{+}}(1+s a+o(s))^{\frac{1}{s}}=e^a=\lim _{s \rightarrow 0^{+}}(1+s a)^{\frac{1}{s}}$ for all $a\neq 0$ (see, \emph{e.g.}, \cite[Lemma 3.1]{HSS22}).
\end{proof}

\begin{lemma}
\label{limit1}
Let $(s_n)_{n\in\N},$ $(p_n)_{n\in\N},$ and $\mu$ as in \eqref{snpn}, then
\begin{align*}
\lim_{n\rightarrow\infty}\left(\frac{2}{p_{n}}\frac{|\varphi_{s_{n}}|_{p_{n}}^{p_{n}}}{\lambda_{1,s_n}|\varphi_{s_{n}}|_{2}^{2}}\right)^{\frac{2}{2-p_{n}}}=\exp\left(-\frac{2\lambda_{1}^{L}}{\mu}-2\int_{\Omega}\ln(|\varphi_{L}|)|\varphi_{L}|^{2}\, dx+1\right).
\end{align*}
\end{lemma}
\begin{proof}
Note that
$
\left(\frac{2}{p_{n}}\right)^{\frac{2}{2-p_{n}}}
=
\left(1-s_n \frac{\mu}{2}+o(s_n)\right)^{\frac{2}{s_n(-\mu+o(1))}}\to e
$ and $(\frac{1}{\lambda_{1,s_n}})^{\frac{2}{2-p_{n}}}\to \exp\left(-{\frac{2\lambda_{1}^{L}}{\mu}}\right)$ as $n\to\infty.$ Moreover,
\begin{align*}
\frac{|\varphi_{s_{n}}|_{p_{n}}^{p_{n}}-|\varphi_{s_{n}}|_{2}^{2}}{s_n}=\frac{p_n-2}{s_n}\int_\Omega \int_0^1 \ln|\varphi_{s_n}| |\varphi_{s_n}|^{2+(p_n-2)\tau}\, d\tau\, dx
\to -\mu \int_\Omega \ln|\varphi_{L}| |\varphi_{L}|^{2}\, dx
\end{align*}
as $n\to\infty$, by dominated convergence, see \cite[Corollary 1.3 and Theorem 1.1 (ii)]{FJW20}. Therefore,
\begin{align*}
\left(\frac{|\varphi_{s_{n}}|_{p_{n}}^{p_{n}}}{|\varphi_{s_{n}}|_{2}^{2}}\right)^{\frac{2}{2-p_{n}}}
&=\left(1-s_n \frac{\mu}{|\varphi_L|_{2}^{2}+o(1)} \int_\Omega \ln|\varphi_{L}| |\varphi_{L}|^{2}\, dx+o(s_n)\right)^{\frac{2}{2-p_{n}}}\\
&\to \exp\left(-\frac{2}{|\varphi_L|_{2}^{2}} \int_\Omega \ln|\varphi_{L}| |\varphi_{L}|^{2}\, dx\right)\qquad \text{ as }n\to\infty.
\end{align*}
Thus, $\left(\frac{2}{p_{n}}\frac{|\varphi_{s_{n}}|_{p_{n}}^{p_{n}}}{\lambda_{1,s_n}|\varphi_{s_{n}}|_{2}^{2}}\right)^{\frac{2}{2-p_{n}}}\to \exp\left(-\frac{2\lambda_{1}^{L}}{\mu}-2\frac{\int_{\Omega}\ln(|\varphi_{L}|)|\varphi_{L}|^{2}\, dx}{|\varphi_{L}|_{2}^{2}}+1\right)$ as $n\to\infty$. The claim follows since $|\varphi_{L}|_{2}^{2}=1$.
\end{proof}

\begin{theorem}
\label{ebounds}
Let $(s_{n})_{n\in\N}$, $(p_n)_{n\in\N},$ $\mu,$ and $(u_{{n}})_{n\in\N}$ as in Theorem \ref{case2}, then
\begin{align*}
\frac{\ln(2)}{2}\exp\left(-\frac{2\lambda_{1}^{L}}{\mu}-2\int_{\Omega}\ln(|\varphi_{L}|)|\varphi_{L}|^{2}\, dx+1\right)|\varphi_{L}|_{2}^{2}+o(1)\leq\|u_{n}\|_{s_{n}}^{2}\leq|\Omega|\exp\left(-\frac{2\lambda_{1}^{L}}{\mu}\right)+o(1)
\end{align*}
as $n\rightarrow\infty$.
\end{theorem}
\begin{proof}
The upper bound follows from Lemma~\ref{limit2} and \eqref{eq:71}. The lower bound follows from \eqref{eq:71}, Lemma~\ref{limit1}, and the fact that
$\lambda_{1,s_n}|\varphi_{s_{n}}|_{2}^{2}
\frac{2^{p_n-2}-1}{p_n-2} \frac{p_n}{2^{p_n}}
\to \frac{\ln2}{2}|\varphi_L|_{2}^{2}=\frac{\ln2}{2}.
$
\end{proof}
\begin{corollary}
\label{norm2bound}
Let $(u_n)_{n\in\N}$ as in Theorem~\ref{ebounds}, then $|u_n|_{2}^{2}\leq|\Omega|\exp\left(-\frac{2\lambda_{1}^{L}}{\mu}\right)+o(1)$ as $n\rightarrow\infty.$
\end{corollary}
\begin{proof}
The result follows from \eqref{eq:poincin} and Theorem~\ref{ebounds}, because $\lambda_{1,s}=1+s_{n}\lambda_{1}^{L}+o(s_{n})$ as $n\rightarrow\infty$ (see \cite[Theorem 1.1]{FJW20}).
\end{proof}

\section{Sublinear power nonlinearity}\label{sec:pnl}

\subsection{Asymptotically linear case}

We characterize first the limiting profile of solutions $u_n$ of \eqref{eq:peq} when $\lim_{n\to\infty}p_n=2$, which we call the asymptotically linear case (because $|t|^{p_n-2}t\to t$ as $n\to\infty$).  We begin our analysis with a study of the least energy solutions of \eqref{s}.

\subsubsection{A logarithmic sublinear problem}

Recall that
\begin{align*}
J_{0}:\mathbb{H}(\Omega)\rightarrow\mathbb{R},\quad J_{0}(u):=\frac{1}{2}\mathcal{E}_{L}(u,u)+I(u),\quad I(u):=\frac{\mu}{4}\int_{\Omega}u^{2}\left(\ln(u^{2})-1\right)\, dx,
\end{align*}
 where $\mu>0.$  This functional is of class $\mathcal{C}^{1}$, see \cite[Lemma 3.1]{HSS22}.  We show first that $J_0$ is coercive. 
 
 \begin{lemma}\label{lem:c}
 $\lim\limits_
 {\substack{\|u\|\to \infty \\ u\in \mathbb H(\Omega)}}
 J_0(u)=\infty$.
 \end{lemma}
\begin{proof}
Let $u\in \mathbb{H}(\Omega)$. By \eqref{eq:bilog}, there is $C=C(\Omega)>0$ such that $\mathcal{E}_{L}(u,u)\geq\|u\|^{2}-C|u|_{2}^{2}.$  Moreover,
\begin{align}
J_{0}(u)\geq\frac{1}{2}\|u\|^{2}-\frac{1}{2}\left(C+\frac{\mu}{2}\right)|u|_{2}^{2}+\frac{\mu}{4}\int_{\Omega}u^{2}\ln(u^{2})\, dx.\label{eq:730}
\end{align}
Let $\widetilde \Omega:=\left\lbrace x\in\Omega: \ln(u^{2}(x))>\frac{2C}{\mu}+1\right\rbrace$. Then, $\frac{\mu}{4}\int_{\widetilde \Omega} u^{2}\ln(u^{2})\, dx \geq  \frac{1}{2}\left(C+\frac{\mu}{2}\right)\int_{\widetilde \Omega}u^{2}\, dx.$ Therefore,
\begin{align*}
J_{0}(u)\geq\frac{1}{2}\|u\|^{2}-\frac{1}{2}\left(C+\frac{\mu}{2}\right)\int_{\Omega\setminus\widetilde \Omega}u^{2}\, dx+\frac{\mu}{4}\int_{\Omega\setminus\widetilde \Omega}u^{2}\ln(u^{2})\, dx.
\end{align*} 
Since $u^{2}\leq e^{\frac{2C}{\mu}+1}$ in $\Omega\setminus\widetilde \Omega$, there is $C_1=C_1(\Omega,\mu)>0$ such that
\begin{align*}
-\frac{1}{2}\left(C+\frac{\mu}{2}\right)\int_{\Omega\setminus\widetilde \Omega}u^{2}\, dx+\frac{\mu}{4}\int_{\Omega\setminus\widetilde \Omega}u^{2}\ln(u^{2})\, dx>-C_1
\end{align*}
and then $J_{0}(u)\geq\frac{1}{2}\|u\|^{2}-C_{1}$, which yields the result.
\end{proof}

\begin{theorem}
\label{exisleast}
For every $\mu>0$ there is a nontrivial unique (up to a sign) least energy solution of 
\begin{align}\label{mu:p}
L_{\Delta}v_{0}=-\mu\ln(|v_{0}|)v_{0}\quad\mbox{in}~\Omega,\quad u_{0}\in \mathbb{H}(\Omega).
\end{align}
Moreover, $v_{0}$ does not change sign.
\end{theorem}
\begin{proof}
By Lemma~\ref{lem:c}, there is a minimizing sequence $(v_{k})_{k\in\N}$ for $J_{0}$, that is, 
$\lim_{k\rightarrow\infty}J_{0}(v_{k})=\inf_{w\in\mathbb{H}(\Omega)}J_{0}(w)=:m.$ By the compact embedding of $\mathbb{H}(\Omega)$ into $L^{2}(\Omega)$, there is $v_{0}\in \mathbb{H}(\Omega)$ such that, up to a subsequence,
\begin{align*}
    v_{{k}}\rightharpoonup v_{0}\quad\text{ in $\mathbb{H}(\Omega)$},
    \qquad 
    v_{{k}}\rightarrow v_{0}\quad \text{ in $L^{2}(\Omega)$},\qquad 
	v_{{k}}\rightarrow v_{0}\quad \text{ a.e. in $\Omega$,}
\end{align*}
as $k\to\infty.$ In particular, $\|v_{0}\|^{2}\leq\liminf\limits_{k\rightarrow\infty}\|v_{k}\|^{2}$.  Moreover, since the function $t \mapsto t^2 \ln t^2$ is bounded below by a constant which is integrable over the bounded set $\Omega$, it follows by Fatou's Lemma that
\begin{align}
\int_{\Omega}v_{0}^{2}\ln(v_{0}^{2})\, dx\leq\liminf\limits_{k\to\infty}\int_{\Omega}v_{k}^{2}\ln(v_{k}^{2})\, dx.\label{eq:734}
\end{align}
Observe that 
\begin{align*}
&\left\lvert \int_{\substack{x,y\in\mathbb{R}^{N}\\ |x-y|\geq 1}}\frac{v_{k}(x)v_{k}(y)}{|x-y|^{N}}\, dx\, dy-\int_{\substack{x,y\in\mathbb{R}^{N}\\ |x-y|\geq 1}}\frac{v_{0}(x)v_{0}(y)}{|x-y|^{N}}\, dx\, dy\right\rvert\\
&\leq\int_{\substack{x,y\in\mathbb{R}^{N}\\ |x-y|\geq 1}}\frac{|v_{k}(x)||v_{k}(y)-v_{0}(y)|}{|x-y|^{N}}\, dx\, dy+\int_{\substack{x,y\in\mathbb{R}^{N}\\ |x-y|\geq 1}}\frac{|v_{0}(y)||v_{k}(x)-v_{0}(x)|}{|x-y|^{N}}\, dx\, dy=:\mathcal{I}_{1}+\mathcal{I}_{2},
\end{align*}
where 
\begin{align*}
\mathcal{I}_{1}&\leq\int_{\mathbb{R}^{N}}|v_{k}(x)|\int_{\mathbb{R}^{N}}|v_{k}(x+y)-v_{0}(x+y)|dydx=\int_{\Omega}|v_{k}(x)|\, dx\int_{\Omega}|v_{k}(y)-v_{0}(y)|dy\rightarrow 0,
\end{align*}
and a similar argument shows that $\mathcal{I}_{2}\to 0$ as $k\to\infty$. Hence, 
\begin{align}
\lim_{k\rightarrow\infty}\int_{\substack{x,y\in\mathbb{R}^{N}\\ |x-y|\geq 1}}\frac{v_{k}(x)v_{k}(y)}{|x-y|^{N}}\, dx\, dy=\int_{\substack{x,y\in\mathbb{R}^{N}\\ |x-y|\geq 1}}\frac{v_{0}(x)v_{0}(y)}{|x-y|^{N}}\, dx\, dy.\label{eq:735}
\end{align}
As a consequence, $J_{0}(v_{0})\leq\liminf\limits_{k\rightarrow\infty}J_{0}(v_{k})=m$ and $v_{0}$ is a least energy solution of $\eqref{eq:727}$. 

To see that $v_0$ is nontrivial, let $\varphi\in C^\infty_c(\Omega)\backslash \{0\}$ and observe that 
\begin{align}\label{mene}
J_0(v_0)=m\leq J_0(t\varphi)=\frac{t^2}{2}\left(\mathcal{E}_L(\varphi,\varphi) + \frac{\mu}{2}\int_\Omega \varphi^2(\ln(t^2)+\ln(\varphi^2)-1)\, dx\right)<0
\end{align}
for $t>0$ sufficiently small, because $\lim\limits_{t\to 0}\ln(t^2)=-\infty$. Therefore $v_0\not\equiv 0$.

By \cite[Lemma 3.3]{CW19}, $\mathcal{E}_{L}(|v_{0}|,|v_{0}|)\leq\mathcal{E}_{L}(v_{0},v_{0})$.  This implies that $\mathcal{E}_{L}(|v_{0}|,|v_{0}|)=\mathcal{E}_{L}(v_{0},v_{0})$, which, by \cite[Lemma 3.3]{CW19}, implies that $v_0$ does not change sign.

Finally, we show the uniqueness (up to a sign) of the least energy solution using a convexity-by-paths argument as in \cite[Section 6]{BFMST18}. Assume, by contradiction, that there are two least-energy solutions $u$ and $v$ such that $u^2\neq v^2$.  Recall that a least-energy solution is a global minimizer of the energy. Let
\begin{align*}
\gamma(t,u,v):=((1-t)u^2+ t v^2)^\frac{1}{2}\qquad \text{ for }t\in[0,1].
\end{align*}
We claim that
\begin{align}\label{claim}
\text{the function $g:[0,1]\to\R$ given by $g(t):=J_0(\gamma(t,u,v))$ is strictly convex in $[0,1]$.}
\end{align}
This would yield a contradiction, since the function $g$ cannot have two global minimizers (at $t=0$ and at $t=1$) and be strictly convex in $[0,1]$.  To see \eqref{claim}, we argue as in \cite[Theorem 6.1]{BFMST18}.

Note that $g(t)=g_1(t)+g_2(t)$, where
\begin{align*}
g_1(t)&:=\mathcal {E}_L(\gamma(t,u,v),\gamma(t,u,v)),\\
g_2(t)&:=\frac{\mu}{2}\int_\Omega [\gamma(t,u,v)(x)]^2(\ln[\gamma(t,u(x),v(x))^2]-1)\, dx.
\end{align*}

First, we show the convexity of $g_1$ in $[0,1]$.  Let $t_1,t_2,\theta\in[0,1]$. We claim that
\begin{align}\label{g1}
g_1((1-\theta)t_1+\theta t_2)\leq (1-\theta)g_1(t_1)+ \theta g_1(t_2).
\end{align}
Indeed, set $U_1:=\gamma(t_1,u,v)
$  and  $U_2:=\gamma(t_2,u,v).$
A direct calculation shows that
\begin{align*}
\gamma((1-\theta)t_1+\theta t_2,u,v)=\gamma(\theta,U_1,U_2).
\end{align*}
Now, for $x,y\in\Omega$, let
\begin{align*}
a=\sqrt{1-\theta}U_1(x),\quad b=\sqrt{1-\theta}U_1(y),\quad c=\sqrt{\theta}U_2(x),\quad d=\sqrt{\theta}U_2(y).
\end{align*}
Then, by the Minkowski inequality, $|(a^2+c^2)^\frac{1}{2}-(b^2+d^2)^\frac{1}{2}|\leq ((a-b)^2+(c-d)^2)^\frac{1}{2},$ which is equivalent to
\begin{align}\label{claim2}
     (\gamma(\theta,U_1,U_2)(x)-\gamma(\theta,U_1,U_2)(y))^2\leq (1-\theta)(U_1(x)-U_1(y))^2+\theta(U_2(x)-U_2(y))^2.
 \end{align}
But then, using \eqref{homega},
\begin{align}
g_1&((1-\theta)t_1+\theta t_2)\notag\\
& =\mathcal {E}_L(\gamma((1-\theta)t_1+\theta t_2,u,v),\gamma((1-\theta)t_1+\theta t_2,u,v))
=\mathcal {E}_L(\gamma(\theta,U_1,U_2),\gamma(\theta,U_1,U_2))\notag\\
&=\frac{c_{N}}{2}\int_\Omega\int_\Omega \frac{(\gamma(\theta,U_1,U_2)(x)-\gamma(\theta,U_1,U_2)(y))^2}{|x-y|^{N}}\, dx\, dy
+\int_\Omega (h_\Omega(x)+\rho_N)\gamma(\theta,U_1,U_2)(x)^2\, dx.\label{g1:1}
\end{align}

By \eqref{claim2},
\begin{align}
 &\int_\Omega\int_\Omega \frac{(\gamma(\theta,U_1,U_2)(x)-\gamma(\theta,U_1,U_2)(y))^2}{|x-y|^{N}}\, dx\, dy\notag\\
 &\leq (1-\theta)\int_\Omega\int_\Omega \frac{(U_1(x)-U_1(y))^2}{|x-y|^{N}}\, dx\, dy+\theta\int_\Omega\int_\Omega \frac{(U_2(x)-U_2(y))^2}{|x-y|^{N}}\, dx\, dy
 \label{g1:2}
\end{align}
and
\begin{align}
\int_\Omega (h_\Omega+\rho_N)\gamma(\theta,U_1,U_2)^2\, dx
=(1-\theta)\int_\Omega (h_\Omega+\rho_N)U_1^2\, dx+\theta\int_\Omega (h_\Omega+\rho_N)U_2^2\, dx.\label{ho0}
 \end{align}
 By \eqref{g1:1}, \eqref{g1:2}, and \eqref{ho0},
 \begin{align*}
  g_1((1-\theta)t_1+\theta t_2)\leq (1-\theta)\mathcal {E}_L(U_1,U_1)+\theta\mathcal {E}_L(U_2,U_2)=(1-\theta)g_1(t_1)+ \theta g_1(t_2),
 \end{align*}
which yields \eqref{g1}.

On the other hand, for $x\in \Omega$, let
\begin{align*}
f(t)&:=[\gamma(t,u,v)(x)]^2(\ln([\gamma(t,u,v)(x)]^2)-1)\\
&=[(1-t)u(x)^2+tv(x)^2](\ln[(1-t)u(x)^2+tv(x)^2]-1).
\end{align*}
Then $f''(t)=\frac{\left(u(x)^2-v(x)^2\right)^2}{(1-t)u(x)^2+t v(x)^2}>0$ in $(0,1),$ whenever $u(x)$ or $v(x)$ are different from zero. Since $u\not\equiv 0$ (see \eqref{mene}), we have that
\begin{align}\label{g2c}
\text{$t\mapsto g_2(t)= \frac{\mu}{2}\int_\Omega \gamma(t,u,v)^2(\ln(\gamma(t,u(x),v(x))^2)-1)\, dx$ is \emph{strictly} convex in $[0,1].$}
\end{align}
By \eqref{g1} and \eqref{g2c}, we conclude that \eqref{claim} must hold, which yields the desired contradiction.
\end{proof}

\subsubsection{Convergence of solutions}
\begin{theorem}
\label{boundH}
Let $(s_k)_{k\in\N}$, $(p_k)_{k\in\N}$, $\mu$, and $(u_k)_{k\in\N}$ as in Theorem \ref{case2}. There is a constant $C=C(\Omega,\mu)>0$ such that $\|u_k\|^2=\mathcal{E}(u_k,u_k)\leq C+o(1)$ as $k\to\infty.$
\end{theorem}
\begin{proof}
By Lemma~\ref{smallorder} we have that $\|u_k\|$ is finite for all $k\in\N$. Fix $k\in\N$ and let $(\varphi_{n})_{n\in\N}\subset\mathcal{C}_{c}^{\infty}(\Omega)$ be such that $\varphi_{n}\rightarrow u_k$ in $\mathcal{H}_{0}^{s_k}(\Omega)$ as $n\to\infty$. We begin with the identity
\begin{align}
\mathcal{I}_{n}:=\frac{\|\varphi_{n}\|_{s_k}^{2}-|\varphi_{n}|_{2}^{2}}{s_k}=\int_{0}^{1}\int_{\mathbb{R}^{N}}|\xi|^{2s_k\tau}\ln(|\xi|^{2})|\widehat{\varphi_{n}}(\xi)|^{2}\, d\xi\, d\tau.\label{eq:723}
\end{align}
From the definition of $J_{s_k}$ (see \eqref{Jdef}) we have that 
\begin{align*}
\mathcal{I}_{n}=\frac{1}{s_k}\left(2J_{s_k}(\varphi_{n})+\frac{2}{p_k}|\varphi_{n}|_{p_k}^{p_k}\right)-\frac{|\varphi_{n}|_{2}^{2}}{s_k}=\frac{1}{s_k}\left( 2J_{s_k}(\varphi_{n})+\left(\frac{2-p_k}{p_k}\right)|\varphi_{n}|_{p_k}^{p_k}\right)+\frac{|\varphi_{n}|_{p_k}^{p_k}-|\varphi_{n}|_{2}^{2}}{s_k},
\end{align*}
and since $u_k$ is a solution of \eqref{eq:peq} and $\varphi_n\to u_k$ in $\cH^s_0(\Omega)$ as $n\to\infty$,
\begin{align*}
2J_{s_k}(\varphi_{n})+\left(\frac{2-p_k}{p_k}\right)|\varphi_n|_{p_k}^{p_k}
=2J_{s_k}(u_k)+\left(\frac{2-p_k}{p_k}\right)|u_k|_{p_k}^{p_k}+o(1)
=o(1)\quad\mbox{as}~n\rightarrow\infty,
\end{align*}
thus, 
\begin{align}
\mathcal{I}_{n}=\frac{|\varphi_{n}|_{p_k}^{p_k}-|\varphi_{n}|_{2}^{2}}{s_k}+o(1)\quad\mbox{as}~n\rightarrow\infty.\label{eq:In}
\end{align}
Observe that, 
\begin{align*}
&\frac{|\varphi_{n}|_{p_k}^{p_k}-|\varphi_{n}|_{2}^{2}}{s_k}=\frac{p_k-2}{s_k}\int_{0}^{1}\int_{\Omega}|\varphi_{n}|^{2+(p_k-2)\tau}\ln(|\varphi_{n}|)\, dx\, d\tau\\
&=\frac{p_k-2}{s_k}\left(\int_{0}^{1}\int_{\left\lbrace|\varphi_{n}|<1\right\rbrace}|\varphi_{n}|^{2+(p_k-2)\tau}\ln|\varphi_{n}|\, dx\, d\tau+\int_{0}^{1}\int_{\left\lbrace|\varphi_{n}|\geq 1\right\rbrace}|\varphi_{n}|^{2+(p_k-2)\tau}\ln|\varphi_{n}|\, dx\, d\tau\right)\\
&\leq \frac{p_k-2}{s_k}\int_{0}^{1}\int_{\left\lbrace|\varphi_{n}|<1\right\rbrace}|\varphi_{n}|^{2+(p_k-2)\tau}\ln|\varphi_{n}|\, dx\, d\tau
\leq  \frac{2-p_k}{s_k}|\Omega| \sup_{t\in(0,1)}|t||\ln|t||<\frac{2-p_k}{s_k}|\Omega|.
\end{align*}
Therefore, by \eqref{eq:In}, we have that 
\begin{align}
\mathcal{I}_{n}\leq \frac{2-p_k}{s_k}|\Omega|+o(1)\quad\mbox{as}~n\rightarrow\infty.\label{eq:724}
\end{align}
On the other hand, 
\begin{align}
\mathcal{I}_{n}&\geq\int_{0}^{1}\int_{\left\lbrace|\xi|<1\right\rbrace}|\xi|^{2s_k\tau}\ln(|\xi|^{2})|\widehat{\varphi}_{n}(\xi)|^{2}\, d\xi\, d\tau+\int_{\left\lbrace|\xi|\geq 1\right\rbrace}\ln(|\xi|^{2})|\widehat{\varphi}_{n}(\xi)|^{2}\, d\xi\notag\\
&=\int_{0}^{1}\int_{\left\lbrace|\xi|<1\right\rbrace}|\xi|^{2s_k\tau}\ln(|\xi|^{2})|\widehat{\varphi}_{n}(\xi)|^{2}\, d\xi\, d\tau-\int_{\left\lbrace|\xi|<1\right\rbrace}|\ln(|\xi|^{2})|\widehat{\varphi}_{n}(\xi)|^{2}\, d\xi+\int_{\mathbb{R}^{N}}\ln(|\xi|^{2})|\widehat{\varphi}_{n}(\xi)|^{2}\, d\xi\notag\\
&=\int_{0}^{1}\int_{\left\lbrace|\xi|<1\right\rbrace}\left(|\xi|^{2s_k\tau}-1\right)\ln(|\xi|^{2})|\widehat{\varphi}_{n}(\xi)|^{2}\, d\xi\, d\tau+\int_{\mathbb{R}^{N}}\ln(|\xi|^{2})|\widehat{\varphi}_{n}(\xi)|^{2}\, d\xi\notag\\
&\geq\int_{\mathbb{R}^{N}}\ln(|\xi|^{2})|\widehat{\varphi}_{n}(\xi)|^{2}\, d\xi=\mathcal{E}_{L}(\varphi_{n},\varphi_{n}).{eq:In}
\end{align}
By \eqref{eq:bilog}, there is $C_{3}=C_{3}(\Omega)>0$ such that $\mathcal{E}_{L}(\varphi_{n},\varphi_{n})\geq\|\varphi_{n}\|^{2}-C_{3}|\varphi_{n}|_{2}^{2}.$ Therefore, \eqref{eq:724}, \eqref{eq:In}, and Proposition \ref{prop1} yield the existence of $C_{4}=C_{4}(\Omega)>0$ such that
\begin{align}\label{vp}
\|\varphi_{n}\|^2\leq
\mathcal{I}_{n}+C_{3}|\varphi_{n}|_{2}^{2}\leq
\frac{2-p_k}{s_k}|\Omega|+C_4+o(1)\qquad \text{as $n\to\infty.$}
\end{align}
Using Lemma~\ref{smallorder} and the fact that $\varphi_n\to u_k$ in $\cH^{s_k}_0(\Omega)$ as $n\to\infty$, taking the limit in \eqref{vp} when $n\rightarrow\infty$ we obtain that $\|u_k\|^2\leq \frac{2-p_k}{s_k}|\Omega|+ C_{4}=(\mu+o(1))|\Omega|+C$ as $k\to\infty$.
\end{proof}

We are ready to show Theorem~\ref{case2}.

\begin{proof}[Proof of Theorem~\ref{case2}]
By Theorem~\ref{boundH}, passing to a subsequence, there is $C=C(\Omega,\mu)>0$ such that, $\|u_n\|\leq C$ for all $n\in\mathbb{N}.$
Then, passing to a further subsequence,
\begin{align}\label{em1}
    u_{{n}}\rightharpoonup u_{0}\quad\text{ in $\mathbb{H}(\Omega)$},
    \qquad 
    u_{{n}}\rightarrow u_{0}\quad \text{ in $L^{2}(\Omega)$},\qquad 
	u_{{n}}\rightarrow u_{0}\quad \text{ a.e. in $\Omega$}
\end{align}
for some $u_0\in \mathbb{H}(\Omega)$.  Let us first show that $u_{0}$ is a non-trivial solution of \eqref{eq:725}. Let $\varphi\in\mathcal{C}_{c}^{\infty}(\Omega)$, by \eqref{eq:726} the identity
\begin{align}
\int_{\Omega}u_{n}(\varphi+s_{n}L_{\Delta}\varphi+o(s_{n}))\, dx&=\int_{\Omega}u_{n}(-\Delta)^{s_{n}}\varphi\, dx=\int_{\Omega}|u_{n}|^{p_n-2}u_{n}\varphi \, dx \notag\\
&=\int_{\Omega}\left(u_{n}+s_{n}\frac{p_n-2}{s_n}\int_{0}^{1}\ln(|u_n|)|u_{n}|^{(p_n-2)\tau}u_{n}\,d\tau\right)\varphi \, dx\label{star}
\end{align}
holds in $L^{\infty}(\Omega)$ for every $n$. Then,  by \eqref{eq:725} and \eqref{star},
\begin{align}
\mathcal{E}_{L}(u_{n},\varphi)+o(1)&=\int_{\Omega}u_{n}L_{\Delta}\varphi \, dx+o(1)=\frac{p_n-2}{s_n}\int_{\Omega}\int_{0}^{1}\ln(|u_{n}|)|u_{n}|^{(p_n-2)\tau}u_{n}\,d\tau\varphi \, dx,\label{eq:728}
\end{align}
as $n\rightarrow\infty$ for all $\varphi\in\mathcal{C}_{c}^{\infty}(\Omega)$. Then, letting $n\rightarrow\infty$ and using Lemma~\ref{d:lemma},
\begin{align}
\mathcal{E}_{L}(u_{0},\varphi)=-\mu\int_{\Omega}\ln(|u_{0}|)u_{0}\varphi\, dx\quad\mbox{for all}~\varphi\in\mathcal{C}_{c}^{\infty}(\Omega).\label{eq:729}
\end{align}
By density, $u_{0}$ is a weak solution of \eqref{eq:727}. Now, let us show that $u_{0}$ is non-trivial. By Theorem~\ref{ebounds}, we know the existence of a positive constant $C=C(\Omega,\mu)>0$ such that
\begin{align*}
C\leq\|u_n\|_{s_n}^2=\int_{\Omega}|u_{n}|^{p_{n}}\, dx\leq|\Omega|^{\frac{2-p_{n}}{2}}\left(\int_{\Omega}|u_{n}|^{2}\, dx\right)^{\frac{p_{n}}{2}},
\end{align*}
and so, $C^{\frac{2}{p_{n}}}|\Omega|^{\frac{p_{n}-2}{p_n}}\leq\int_{\Omega}|u_{n}|^{2}\, dx. $ Letting $n\rightarrow\infty$ we conclude that  $0<C\leq\int_{\Omega}|u_{0}|^{2}\, dx.$ Therefore, $u_{0}\neq 0$. Since $u_{0}$ is a weak solution of \eqref{eq:727}, we have that 
\begin{align*}
J_{0}(u_{0})=\frac{\mathcal{E}_{L}(u_{0},u_{0})}{2}+\frac{\mu}{4}\int_{\Omega}u_{0}^{2}\left(\ln(u_{0}^{2})-1\right)\, dx=-\frac{\mu}{4}\int_{\Omega}u_{0}^{2}\, dx.
\end{align*}
To see that $u_0$ is of least energy it remains to show that $-\frac{\mu}{4}|u_{0}|_{2}^{2}=\inf_{\mathbb{H}(\Omega)}J_{0}.$ By H\"older's inequality,
\begin{align*}
0\leq\limsup_{n\rightarrow\infty}|u_{n}-u_{0}|_{p_{n}}\leq\limsup_{n\rightarrow\infty} |\Omega|^{\frac{2-p_{n}}{2p_{n}}}|u_{n}-u_{0}|_{2}=0,
\end{align*}
thus, using Proposition~\ref{prop1} and Lemma~\ref{lem2.2}, $\lim\limits_{n\rightarrow\infty}\|u_{n}\|_{s_{n}}^{2}=\lim\limits_{n\rightarrow\infty}|u_{n}|_{p_n}^{p_n}=|u_{0}|_{2}^{2}.$ Then,
\begin{align}
-\frac{\mu}{4}\lim_{n\rightarrow\infty}\|u_{n}\|_{s_{n}}^{2}=-\frac{\mu}{4}\lim_{n\rightarrow\infty}|u_{n}|_{p_n}^{p_n}=-\frac{\mu}{4}|u_{0}|_{2}^{2}=J_{0}(u_{0}).
\label{eq:736}
\end{align}
On the other hand, by Theorem~\ref{exisleast}, there is $v_{0}\in\mathbb{H}(\Omega)$ such that $J_{0}(v_{0})=\inf_{\mathbb{H}(\Omega)}J_{0}$ and by \cite[Theorem 3.1]{CW19} there is a sequence $(v_{k})_{k\in\N}\subset\mathcal{C}_{c}^{\infty}(\Omega)$ such that $v_{k}\rightarrow v_{0}$ in $\mathbb{H}(\Omega)$ as $k\rightarrow\infty$. Since $v_{k}\in\mathcal{C}_{c}^{\infty}(\Omega)$ for all $k\in\mathbb{N}$ and $u_n$ is of least energy (by uniqueness \cite[Theorem 6.1]{BFMST18}), we have that
\begin{align*}
-\frac{\mu}{4}\lim_{n\rightarrow\infty}\|u_{n}\|_{s_{n}}^{2}=\lim_{n\rightarrow\infty}\frac{1}{s_{n}}J_{s_{n}}(u_{n})\leq\lim_{n\rightarrow\infty}\frac{1}{s_{n}}J_{s_{n}}(v_{k}).
\end{align*}
By \eqref{eq:weakaprox}, we obtain the following inequality
\begin{align}
-\frac{\mu}{4}\lim_{n\rightarrow\infty}\|u_{n}\|_{s_{n}}^{2}
&\leq-\frac{\mu}{4}|v_{k}|_{2}^{2}+\frac{1}{2}\left(\mathcal{E}_{L}(v_{k},v_{k})+\mu\int_{\mathbb{R}^{N}}|v_{k}|^{2}\ln|v_{k}|\, dx\right)\notag\\
&=-\frac{\mu}{4}|v_{0}|_{2}^{2}+o(1)
=J_0(v_0)+o(1)=\inf_{\mathbb{H}(\Omega)}J_{0}+o(1)
\label{eq:743}
\end{align}
as $k\rightarrow\infty$, according with Lemma~\ref{haprox}. Therefore, by \eqref{eq:736} and \eqref{eq:743},
\begin{align*}
\inf_{\mathbb{H}(\Omega)}J_{0}\leq J_{0}(u_{0})=-\frac{\mu}{4}|u_{0}|^{2}_{2}
=-\frac{\mu}{4}\lim_{n\rightarrow\infty}\|u_{n}\|_{s_{n}}^{2}\leq\inf_{\mathbb{H}(\Omega)}J_{0}
\end{align*}
as claimed.  Since $u_{0}\in\mathbb{H}(\Omega)$ is a least energy solution of \eqref{eq:727}, Theorem~\ref{exisleast} implies that $u_{0}$ does not change sign in $\Omega$.

To conclude the proof, we show that $u_0\in L^\infty(\Omega)$ and 
\begin{align}\label{C0cl}
|u_0|_\infty\leq ((2\operatorname{diam}(\Omega))^{2}e^{\frac{1}{2}-\rho_{N}})^{\frac{1}{\mu}}=:C_0.
\end{align}
By Proposition~\ref{prop1}, $|u_{n}|_{\infty}\leq C_0+o(1)$ as $n\to\infty$. Assume, by contradiction, that there is $\eps>0$ and set $\omega\in \Omega$ of positive measure such that $|u_{0}|>(1+\eps)C_0$ in $\omega$. This implies that 
\begin{align*}
|u_{n}(x)-u_{0}(x)|\geq|u_{0}(x)|-|u_{n}(x)|>(1+\eps)C_0-C_{0}=\eps C_0\quad \text{ for a.e. } x\in\omega.
\end{align*}
Thus, $\int_{\Omega}|u_{n}-u_{0}|^{2}\, dx\geq\int_{\omega}|u_{n}-u_{0}|^{2}\, dx>\eps C_0|\omega|>0,$ which contradicts the $L^{2}$-convergence of $u_{n}$ to $u_{0}$. Therefore, \eqref{C0cl} holds. In consequence, up to a subsequence, the convergence $u_{n}\rightarrow u_{0}$ in $L^{q}(\Omega)$ for any $1\leq q<\infty$ now follows by the dominated convergence theorem. Finally, since \eqref{eq:727} has a unique least energy solution, we have that the limit $u_0$ is independent of the chosen subsequence of $(u_n)_{n\in\N}$, therefore the whole sequence $(u_n)_{n\in\N}$ must also converge to $u_0$ in $L^2(\Omega)$.
\end{proof}

\begin{remark}\label{rmk:r:1}
One could also phrase the statement of Theorem~\ref{case2} as follows: Let $\Omega\subset \R^N$ be an open bounded Lipschitz set.  Let $h:(0,1)\to (0,1)$ be a function such that $h(s)/s\to \mu\in(0,\infty)$ as $s\to 0^+$. For $s\in(0,1)$, let $u_s$ be the unique positive solution of
\begin{align*}
    (-\Delta)^{s} u_s = u_s^{1-h(s)}\quad \text{ in }\Omega,\qquad u_s=0\quad \text{ on }\R^N\backslash \Omega.
\end{align*}
Then $u_{s}\rightarrow u_{0}$ in $L^{q}(\mathbb{R}^{N})$ as $s\rightarrow 0^+$ for all $1\leq q<\infty,$ where $u_0\in\mathbb{H}(\Omega)\cap L^{\infty}(\Omega)\backslash \{0\}$ is the unique nonnegative least energy solution of \eqref{eq:727}.
\end{remark}

Since the nonlinearity $-\mu\ln|u|u$ can change sign even if $u\geq 0$, one cannot use  standard maximum principles to characterize the sign properties of the solution; however, in the next result we show a strong maximum principle for continuous weak solutions of \eqref{mu:p} by working on small neighborhoods and using the negative sign of $-\mu$.
\begin{lemma}\label{strong:mp}
Let $v\in\mathcal{C}(\mathbb{R}^{N})$ be a nontrivial nonnegative weak solution of \eqref{mu:p}, then $v>0$ in $\Omega$.
\end{lemma}
\begin{proof}
By contradiction, assume that there is $x_{0}\in\Omega$ such that 
\begin{align}\label{c0}
v(x_{0})=0.    
\end{align}
 By continuity and because $v\neq 0$, there are $\delta>0$, an open set $V\subset\{x\in\Omega~:~v(x)>\delta\}$, and $r>0$ such that 
$-\mu\ln|v|v\geq 0$  in $B_{r}(x_{0})$ and $\operatorname{dist}(B_{r}(x_{0}),V)>0.$ 

By \cite[Corollary 1.9]{CW19}, we can consider, if necessary, $r$ smaller so that $L_{\Delta}$ satisfies the weak maximum principle in $B_{r}(x_{0})$ and $\lambda_1^L>0$, where $\lambda_1^L$ is the first eigenvalue of $L_\Delta$. Now, a standard application of the Riesz representation theorem yields the existence of a unique solution $\tau\in\mathbb{H}(\Omega)$ of
\begin{align*}
L_{\Delta}\tau=1\quad\mbox{in}~B_{r}(x_{0}),\qquad
\tau=0\quad\mbox{in}~\mathbb{R}^{N}\setminus B_{r}(x_{0}).
\end{align*}
Moreover, by \cite[Theorem 1.1]{CS22}, we know that $\tau$ is a classical solution, namely, that $L_\Delta \tau(x)=1$ holds pointwisely for $x\in\Omega$. This implies that $\tau>0$ in $B_{r}(x_{0})$, since if $\tau(y_0)=0$ for some $y_0\in \Omega$, then
\begin{align*}
1=L_\Delta \tau(y_0)=-c_N\int_{B_r(x_0)}\frac{\tau(y)}{|y_0-y|^N}\, dy<0,
\end{align*}
which would yield a contradiction.  Now we argue as in \cite{JF}. Let $\chi_{V}$ denote the characteristic function of $V$ and note that, for $x\in B_{r}(x_{0})$, $\chi_{V}(x)=0$ and therefore 
\begin{align*}
L_{\Delta}\chi_{V}(x)=-c_{N}\int_{\mathbb{R}^{N}}\frac{\chi_{V}(y)}{|x-y|^{N}}\, dy=-c_{N}\int_{V}\frac{1}{|x-y|^{N}}\, dy\leq -c_{N}|V|\inf_{z\in B_r(x_0)}(|z-y|^{-N}).
\end{align*}
Let $K:=c_{N}|V|\inf_{z\in B_r(x_0)}(|z-y|^{-N})$ and $\varphi:=\frac{K}{2}\tau+\chi_{V}$. Then, $L_{\Delta}\varphi\leq K/2-K\leq 0$ in $B_{r}(x_{0})$. Moreover, since $v>\delta$ in $V$, we have that 
\begin{align}
L_{\Delta}(v-\delta\varphi)\geq 0~\mbox{in}~B_{r}(x_{0}),\qquad 
v-\delta\varphi\geq 0~\mbox{in}~\mathbb{R}^{N}\setminus B_{r}(x_{0})
\end{align}
in the weak sense. Then, by the weak maximum principle (see \cite[Corollary 1.8]{CW19}) we obtain that $v\geq\delta\varphi\geq \delta\tau>0$ in $B_{r}(x_{0}),$ a contradiction to \eqref{c0}.  Therefore $v>0$ in $\Omega$.
\end{proof}

\begin{proof}[Proof of Theorem~\ref{exisleast:intro}]
Existence and uniqueness of least energy solutions follow from Theorem~\ref{exisleast}, and the estimate \eqref{C0clthm} follows from \eqref{C0cl}, by uniqueness.  Assume now that $\Omega$ satisfies a uniform exterior sphere condition, then, since $v\in L^\infty(\Omega)$, it follows that $\ln|v|v\in L^\infty(\Omega)$, and, by \cite[Theorem 1.11]{CW19}, we have that $v\in C(\overline{\Omega})$.   The estimate \eqref{lhr} follows from \cite[Corollary 5.8]{CS22} and a standard density argument. The fact that $|v|>0$ in $\Omega$ follows from Lemma~\ref{strong:mp}.
\end{proof}

\begin{remark}\label{reg:rmk}
Note that the regularity in \eqref{lhr} is not enough to guarantee that $u$ is a classical solution, namely, that $L_\Delta u$ can be evaluated pointwisely. This would require a refinement of \cite[Theorem 1.1]{CS22}, see \cite[Section 6, open problem (1)]{CS22}. \end{remark}

\subsection{Asymptotically sublinear case}

Now we focus our attention on the analysis of solutions $u_n$ of \eqref{eq:peq} when $\lim_{n\to\infty}p_n\in[1,2)$, which we call the asymptotically sublinear case.  We begin by considering an auxiliary nonlinear eigenvalue problem in a rescaled domain. Let $(s_n)\subset (0,1)$ be such that $\lim_{n\to \infty}s_n=0$, 
\begin{align*}
p_{n}\subset (1,2)\quad \text{ be such that }\quad \lim_{n\to \infty}p_n = p\in [1,2).     
\end{align*}
Let $\lambda:=|\Omega|$ and $\Omega_{\lambda}:=\frac{1}{\lambda}\Omega$ (note that $|\Omega_{\lambda}|=1$). Set
\begin{align}
\Lambda_{0}&:=\inf\left\lbrace \int_\Omega|v|^{2}\, dx \::\: v\in L^2(\Omega_{\lambda})\quad \text{ and }\quad \int_{\Omega_\lambda}|v|^{p}\, dx=1\right\rbrace,\label{l0}\\
\Lambda_{n}&:=\inf\left\lbrace \|v\|_{s_{n}}^{2}:v\in\mathcal{H}_{0}^{s_{n}}(\Omega_{\lambda}),\ |v|^{p_n}_{p_{n}}=1\right\rbrace,
\end{align}
and let $\chi_{\Omega_\lambda}$ denote the characteristic function of $\Omega_\lambda$.

\begin{lemma}\label{lem:lam0} The infimum $\Lambda_0$ is achieved at $\chi_{\Omega_\lambda}$; in particular, $\Lambda_0=1=|\chi_{\Omega_\lambda}|_2^2$.
\end{lemma}
 \begin{proof}
 Clearly, $\Lambda_0\leq 1$, because $|\Omega_\lambda|=1=|\chi_{\Omega_{\lambda}}|_2^2=|\chi_{\Omega_{\lambda}}|_{p}^{p}$. On the other hand, for each $v\in \left\lbrace v\in L^2(\Omega_{\lambda})\::\: v=0\text{ in }\R^N\backslash \Omega_\lambda \text{ and }|v|^p_{p} = 1\right\rbrace$ it holds that
$1=|v|^p_{p}\leq|v|_{2}^{p}$, thus $1\leq \Lambda_{0}$.
 \end{proof}

\begin{proposition}\label{prop:Lamn}
For every $n\in\mathbb N$ there is $v_n\in \mathcal{H}_{0}^{s_{n}}(\Omega_{\lambda})$ such that $\Lambda_n=\|v_n\|_{s_{n}}^{2}$. Moreover, $v_n\to 1$ in $L^2(\Omega_\lambda)$, $\Lambda_n\to 1$ as $n\to \infty$, and $(v_{n})_{n\in\N}$ is a minimizing sequence for $\Lambda_{0}$.
\end{proposition}
\begin{proof}
Using the compact embedding of $\cH^{s_n}_0(\Omega_\lambda)$ into $L^{p_n}(\Omega_\lambda)$ and standard arguments, we have that the infimum $\Lambda_n$ is achieved at some nontrivial $v_n\in \cH^{s_n}_0(\Omega_\lambda)$. We can assume w.l.o.g. that $v_n$ is nonnegative.  By the Lagrange multiplier theorem, each $v_n$ is a solution of
\begin{align}\label{vn:eq}
    (-\Delta)^{s_n}v_n=\Lambda_{n} v_n^{p_n-1}\quad \text{ in }\Omega_\lambda,\qquad v_n\in \cH^{s_n}_0(\Omega_\lambda).
\end{align}
Let $\varphi\in C^\infty_c(\Omega_{\lambda})\backslash\{0\}$ and recall that $\lim_{n\to \infty}p_n = p\in [1,2)$, then
\begin{align*}
\Lambda_{n}
=\|v_{n}\|_{s_{n}}^{2}
\leq\frac{\|\varphi\|_{s_{n}}^{2}}{|\varphi|^2_{p_n}}
=\frac{|\varphi|_{2}^{2}}{|\varphi|^2_{p}}+o(1)
\qquad \text{ as }n\to \infty,
\end{align*}
because $|\varphi|_{p_n}\rightarrow|\varphi|_{p}$ and $\|\varphi\|_{s_{n}}^{2}\rightarrow|\varphi|_{2}^{2}$ as $n\rightarrow\infty$. Thus, passing to a subsequence, $\Lambda_{n}=\|v_{n}\|_{s_{n}}^{2}\rightarrow\Lambda_{0}^{\ast}$ as $n\to\infty$ for some $\Lambda_0^*\geq 0$.  Observe that 
\begin{align}
\Lambda_{0}^{\ast}\leq\frac{|\varphi|_{2}^{2}}{|\varphi|^2_{p}}\qquad\text{ for all }\varphi\in\mathcal{C}_{c}^{\infty}(\Omega_{\lambda})\backslash \{0\}.\label{eq:113}
\end{align}
Let $\Lambda_0$ be as in \eqref{l0}.  By Lemma \ref{lem:lam0}, \eqref{eq:113}, and the density of $C^\infty_c(\Omega_\lambda)$ in $L^2(\Omega)$,
\begin{align*}
\Lambda_{0}^{\ast}
\leq\Lambda_{0}
\leq\frac{|v_{n}|_{2}^{2}}{|v_{n}|_{p}^{2}}
\leq \lambda_{1,s_n}\frac{\|v_{n}\|_{s_{n}}^{2}}{|v_{n}|_{p}^{2}}
=(1+o(1))\frac{\Lambda_n}{|v_{n}|_{p}^{2}},
\end{align*}
as $n\to\infty$, where we have used that $1+o(1)=\lambda_{1,s_n}:=\inf\{\|v\|_{s_n}^2\::\: v\in\cH^{s_n}_0(\Omega_\lambda)\text{ and }|v|_2=1\}$ as $n\to\infty$, see \cite[Theorem 1.1]{FJW20}.  Notice that, by Proposition~\ref{prop1}, the sequence $(v_{n})_{n\in\N}$ is uniformly bounded in $L^{\infty}(\Omega_{\lambda})$. Thus, Lemma~\ref{lem2.2} yields that $\left\lvert\int_{\Omega}|v_{n}|^{p}-\int_{\Omega}|v_{n}|^{p_{n}}\right\rvert=o(1)$ as $n\rightarrow\infty$ and, since $|v_{n}|_{p_{n}}=1$, $    \lim\limits_{n\rightarrow\infty}|v_{n}|_{p}=1.$  Then $\Lambda_{0}\leq\Lambda_{0}^{\ast}$ and therefore $\Lambda_{0}=\Lambda_{0}^{\ast}$, namely,
\begin{align}
\|v_{n}\|_{s_{n}}^{2}=\Lambda_{n}\rightarrow\Lambda_{0}\quad\mbox{as}\quad n\rightarrow\infty.\label{eq:114}
\end{align}
Now, since
$\Lambda_{0}\leq\frac{|v_{n}|_{2}^{2}}{|v_{n}|_{p}^{2}}\leq \frac{\|v_{n}\|_{s_{n}}^{2}}{|v_{n}|_{p}^{2}}\lambda_{1,s_n}^{-1},$
\begin{align*}
\left(1+o(1)\right)\Lambda_{0}=|v_{n}|_{p}^{2}\Lambda_{0}\leq |v_{n}|_{2}^{2}\leq (\lambda_{1,s_{n}})^{-1}\Lambda_{n}=\left(1+o(1)\right)\left(\Lambda_{0}+o(1)\right)
\end{align*}
as $n\to\infty$. As a consequence, $v_{n}$ is a minimizing sequence for $\Lambda_{0}$, namely, 
\begin{align}\label{eq:115}
|v_{n}|_{2}^{2}\rightarrow \Lambda_{0}\quad\mbox{as}\quad n\rightarrow \infty.
\end{align}

Finally, we show that $v_n\to 1$ in $L^2(\Omega_\lambda)$ as $n\to\infty.$  By Lemma~\ref{lem:lam0} we have that $\Lambda_{0}=1.$ By contradiction, assume that there is $\delta>0$ and $n_0\in\mathbb N$ such that $\int_{\Omega_{\lambda}}|v_{n}-1|^{2}\,dx\geq\delta>0$ for all $n\geq n_0.$ Then, using \eqref{eq:115},
\begin{align}\label{eq:116}
\int_{\Omega_{\lambda}}v_{n}\,dx\leq 1-\frac{\delta}{2}+o(1)\qquad \text{ as }n\to\infty.
\end{align}
Let $\alpha_{n}:=2(p_{n}-1),$ $\beta_{n}:=2-p_{n},$ $r_{n}:=\frac{2}{\alpha_{n}},$ $q_{n}:=\frac{1}{\beta_{n}}.$ Notice that $r_{n},q_{n}>1$ for all $n\in\mathbb{N}$, $\tfrac{1}{r_{n}}+\tfrac{1}{q_{n}}=1$ and $\alpha_{n}+\beta_{n}=p_{n}$. Then, by Young's inequality,
\begin{align}\label{yi}
1=|v_{n}|_{p_{{n}}}^{p_{{n}}}=\int_{\Omega_{\lambda}}v_{n}^{\alpha_{n}}v_{n}^{\beta_{n}}\, dx\leq (p_{n}-1)|v_{n}|_{2}^{2}+(2-p_{n})|v_{n}|_{1}.
\end{align}
Then by \eqref{eq:115}, \eqref{eq:116}, \eqref{yi},
\begin{align*}
1&\leq(p_{n}-1)\left(1+o(1)\right)+(2-p_{n})\left(1+o(1)-\frac{\delta}{2}\right)\\
&=(p-1+o(1))\left(1+o(1)\right)+(2-p+o(1))\left(1+o(1)-\frac{\delta}{2}\right)
=1-\frac{2-p}{2}\delta+o(1)
\end{align*}
as $n\to\infty$ and the contradiction follows.
\end{proof}

We are ready to show Theorem~\ref{main:power}.

\begin{proof}[Proof of Theorem~\ref{main:power}]
Let $u_n\in \cH^{s_n}_0(\Omega)$ be the positive least-energy solution of \eqref{eq:peq} and let $w_{n}(x):=\lambda^{-\frac{2s_{n}}{2-p_{n}}}u_{n}(\lambda x)$.  Then $w_n$ is a positive least-energy solution of 
\begin{align}\label{wn:eq}
(-\Delta)^{s_{n}}w_{n}=|w_{n}|^{p_{n}-2}w_{n},\qquad w_{n}\in\mathcal{H}_{0}^{s_{n}}(\Omega_{\lambda}),
\end{align}
$\Omega_\lambda=\frac{\Omega}{|\Omega|}$, and $\|w_{n}\|_{s_{n}}
=\lambda^{-\frac{2s_{n}}{2-p_{n}}}\lambda^{\frac{2s_{n}-N}{2}}\|u_{n}\|_{s_{n}}=\lambda^{\frac{p_{n}N-2p_{n}s_{n}-2N}{2(2-p_{n})}}\|u_{n}\|_{s_{n}}.$ 
Passing to a subsequence, let $v_n$ be the minimizers of $\Lambda_n$ given in Proposition~\ref{prop:Lamn}. By uniqueness of positive solutions of sublinear problems (see \emph{e.g.} \cite[Theorem 6.1]{BFMST18}), the equations \eqref{vn:eq} and \eqref{wn:eq} imply that $w_{n}=\Lambda_{n}^{\frac{1}{p_{n}-2}}v_{n}.$
Then, by Proposition~\ref{prop:Lamn} and Lemma~\ref{lem:lam0}, $\lambda^{-\frac{2s_{n}}{2-p_{n}}}u_{n}(\lambda x)= w_{n}\rightarrow 1$ in  $L^{2}(\Omega_{\lambda})$ as $n\to\infty$. Since $\lim_{n\to\infty}p_n\in(1,2)$, we conclude that $u_{n}\rightarrow 1$ in $L^{2}(\Omega)$ as $n\rightarrow\infty$, as claimed.  The convergence in $L^q(\Omega)$ for $1\leq q<\infty$ now follows from Proposition~\ref{prop1} and the dominated convergence theorem.  Note that the limit $1$ is independent of the chosen subsequence of $(u_n)_{n\in\N}$, therefore the whole sequence $(u_n)_{n\in\N}$ must also converge to $1$ in $L^q(\Omega)$ for $1\leq q<\infty$.  This ends the proof.
\end{proof}

\begin{remark}\label{rmk:r:2}
One could also phrase the statement of Theorem \ref{main:power} as follows: Let $\Omega\subset \R^N$ be an open bounded Lipschitz set, $h:(0,1)\to (0,1)$ be a function such that $h(s)\to p$ as $s\to 0^+$ for some $p\in[0,1)$ and, for $s\in(0,1)$, let $u_s$ be the unique positive solution of
\begin{align*}
    (-\Delta)^{s} u_s = u_s^{h(s)}\quad \text{ in }\Omega,\qquad u_s=0\quad \text{ on }\R^N\backslash \Omega.
\end{align*}
Then $u_{s}\rightarrow 1$ in $L^{q}(\mathbb{R}^{N})$ as $s\rightarrow 0^+$ for all $1\leq q<\infty$.
\end{remark}

\section{Other sublinear-type problems}\label{gen:sec}

Recall that $\Omega\subset\mathbb{R}^{N}$ is an open bounded Lipschitz set. In this section, $(s_{n})_{n\in\N}$ is a sequence in $(0,1)$ such that $\lim_{n\to\infty}s_n=0$. Let $\Omega\subset\mathbb{R}^{N}$ be a bounded open set with Lipschitz boundary, and let
\begin{align}\label{eq:ac00}
\varepsilon>0,\qquad A>0,\qquad r>2.
\end{align}
Define 
\begin{align}\label{Lndef}
 L_{n}(u):=\tfrac{1}{2}\|u\|_{s_n}^2+\tfrac{A}{r}|u|^{r}_{r},\qquad \Sigma_{n}:=\left\lbrace v\in\mathcal{H}_{0}^{s}(\Omega)\cap L^{r}(\Omega)\::\:|\Omega|^{-1}\varepsilon|u|_{2}^{2}=1\right\rbrace,
\end{align}
and consider the following variational problem
\begin{align}\label{thnn}
    \Theta_n:=\inf\left\{L_n(u)\::\: u\in \Sigma_{n}\right\}.
\end{align}
Using the compact embedding $\mathcal{H}_{0}^{s}(\Omega)\hookrightarrow L^{2}(\Omega)$ and standard arguments, it follows that the infimum $\Theta_n$ is achieved at a non-trivial function $v_{n}\in \Sigma_n$ which does not change sign (since $\mathcal{E}_{s}(|v_{n}|,|v_{n}|)\leq\mathcal{E}_{s}(v_{n},v_{n})$).  Throughout this section we assume that
\begin{align}\label{vn}
v_{n}\in\Sigma_{n} \text{ is a non-negative function such that }\Theta_{n}=L_{n}(v_{n}).
\end{align}

\subsection{Auxiliary nonlinear eigenvalue problems}

Let $\eps>0,$ $A>0$, $r>2$, define $G(u):=|\Omega|^{-1}\varepsilon\int_{\Omega}|u|^{2}\, dx$ and 
\begin{align}\label{J00def}
 J(u):=\tfrac{1}{2}|u|_{2}^2+\tfrac{A}{r}|u|^{r}_{r},\qquad 
 \Sigma_{0}:=\left\lbrace v\in L^2(\R^N)\cap L^{r}(\R^N)\::\:u=0\text{ in }\R^N\backslash \Omega,\ G(u)=1\right\rbrace,
\end{align}

\begin{theorem}
\label{infimum}
Let $\Omega\subset\mathbb{R}^{N}$ be an open bounded Lipschitz set. Let $\Theta_{0}:=\inf\left\lbrace J(u)\::\: u\in \Sigma_0\right\rbrace.$ Then, $\Theta_{0}=\frac{|\Omega|}{2\varepsilon}+\frac{A|\Omega|}{r\varepsilon^{r/2}}.$
\end{theorem}
\begin{proof}
Since $\varepsilon^{-1/2}\chi_{\Omega}\in \Sigma_0$, we have that $\Theta_{0}\leq\frac{|\chi_{\Omega}|_{2}^{2}}{2\varepsilon}+\frac{A|\chi_{\Omega}|_{r}^{r}}{r\varepsilon^{r/2}}=\frac{|\Omega|}{2\varepsilon}+\frac{A|\Omega|}{r\varepsilon^{r/2}}.$ On the other hand, for every $u\in L^{r}(\Omega)$ such that $\tfrac{\varepsilon|u|_{2}^{2}}{|\Omega|}=1$, H\"older's inequality yields that $\frac{|\Omega|}{\varepsilon^{r/2}}
\leq|u|_{r}^{r}.$ Then, by \eqref{eq:ac00}, $\frac{|\Omega|}{2\varepsilon}+\frac{A|\Omega|}{r\varepsilon^{r/2}}\leq\frac{|u|_{2}^{2}}{2}+\frac{A|u|_{r}^{r}}{r},$ holds for all $u\in \Sigma_0$. This proves the result.
\end{proof}

\begin{theorem}
\label{acaprox}
Let $\Omega\subset\mathbb{R}^{N}$ be an open bounded Lipschitz set. Then 
\begin{align}
\Theta_{n}\rightarrow\Theta_{0}\quad\mbox{as}~n\rightarrow\infty\label{eq:ac1}
\end{align}
and $(v_{n})_{n\in\N}$ is a minimizing sequence for $\Theta_{0}$, that is 
\begin{align}
\frac{|v_{n}|_{2}^{2}}{2}+\frac{A|v_{n}|_{r}^{r}}{r}\rightarrow\Theta_{0}\quad\mbox{as}~n\rightarrow\infty.\label{eq:ac2}
\end{align}
\end{theorem}
\begin{proof}
Let $\varphi\in\mathcal{C}_{c}^{\infty}(\Omega)\setminus\left\lbrace 0 \right\rbrace$ and set $\phi:=\left(\tfrac{|\Omega|}{\varepsilon}\right)^{1/2}\tfrac{\varphi}{|\varphi|_{2}}$ so that $|\phi|_{2}^{2}=\tfrac{|\Omega|}{\varepsilon}$. Then, 
\begin{align*}
\Theta_{n}&=\frac{\|v_{n}\|_{s_{n}}^{2}}{2}+\frac{A|v_{n}|^{r}_{r}}{r}\leq\frac{\|\phi\|_{s_{n}}^{2}}{2}+\frac{A|\phi|^{r}_{r}}{r}=\frac{|\phi|_{2}^{2}}{2}+\frac{A|\phi|^{r}_{r}}{r}+o(1)
=\frac{|\Omega|}{2\varepsilon}+\frac{A|\phi|^{r}_{r}}{r}+o(1)
\end{align*} 
as $n\rightarrow\infty$, where $(s_n)_{n\in\N}\subset (0,1)$ is the sequence associated to $\Theta_n$. Then, up to a subsequence,
$\Theta_{n}=\frac{\|v_{n}\|_{s_{n}}^{2}}{2}+\frac{A|v_{n}|^{r}_{r}}{r}\rightarrow\Theta_{0}^{\ast}$ as $n\rightarrow\infty$ for some $\Theta_{0}^{\ast}\geq 0$. In particular, it holds that
\begin{align*}
\Theta_{0}^{\ast}\leq\frac{|\Omega|}{2\varepsilon}+\frac{A}{r}\left(\frac{|\Omega|}{\varepsilon}\right)^{r/2}\frac{|\varphi|_{r}^{r}}{|\varphi|_{2}^{r}} \qquad \text{ for all }\varphi\in\mathcal{C}_{c}^{\infty}(\Omega)\setminus\left\lbrace 0\right\rbrace.
\end{align*}
Using the definition of $\Theta_0$ and a density argument, it follows that
\begin{align}
\Theta_{0}^{\ast}\leq\Theta_{0}.\label{eq:ac3}
\end{align}
On the other hand, using that $v_{n}\in L^{r}(\Omega)$ and $|v_{n}|_{2}^{2}=|\Omega|\varepsilon^{-1}$ for all $n\in\mathbb{N}$, together with \eqref{eq:poincin},
\begin{align}
\Theta_{0}\leq\frac{|v_{n}|_{2}^{2}}{2}+\frac{A|v_{n}|^{r}_{r}}{r}\leq\frac{(\lambda_{1,s_n})^{-1}\|v_{n}\|_{s_{n}}^{2}}{2}+\frac{A|v_{n}|^{r}_{r}}{r},\label{eq:ac4}
\end{align}
implying that $\Theta_{0}\leq\Theta_{n}+o(1)=\Theta_{0}^{\ast}+o(1)$ as $n\rightarrow\infty.$ This inequality combined with \eqref{eq:ac3} yields \eqref{eq:ac1}. Then, by \eqref{eq:ac4}, $\Theta_{0}\leq\frac{|v_{n}|_{2}^{2}}{2}+\frac{A|v_{n}|^{r}_{r}}{r}
=\Theta_n+o(1)
=\Theta_{0}+o(1)$ as $n\rightarrow\infty,$ which proves \eqref{eq:ac2}.
\end{proof}

The following result characterizes the minimizer of $\Theta_0$.

\begin{theorem}
\label{acmin}
Let $J$, $\Sigma_0$, and $G$ be as in \eqref{J00def}. If $u\in \Sigma_0$ is a minimizer for $\Theta_0$, then  $|u|=\varepsilon^{-1/2}$ a.e. in $\Omega$.
\end{theorem}
\begin{proof}
Clearly, both $J$ and $G$ are differentiable on $L^{r}(\Omega)$.  Assume that $u\in \Sigma_0$ is a minimizer for $\Theta_0$. Since $u\neq 0$, there is a test function $\varphi_{u}\in\mathcal{C}_{c}^{\infty}(\Omega)$ such that $D_{\varphi_{u}}G(u)=2|\Omega|^{-1}\varepsilon\int_{\Omega}u\varphi_{u} dx\neq 0,$ where $D_{\varphi_{u}}G(u)$ is the Gâteaux derivative of $G$ at $u$ in the direction $\varphi_u$.

Then, by the Lagrange multiplier Theorem (see, for example, \cite[Chap. 2, Sec. 1, Theorem~1]{giaquinta}), there is a real number $\lambda_{M}$ such that the equation $D_{\varphi}J(u)-\lambda_{M}D_{\varphi}G(u)=0$ holds for all $\varphi\in\mathcal{C}_{c}^{\infty}(\Omega)$, that is,
\begin{align*}
\int_{\Omega}\left(u+A|u|^{r-2}u-2\lambda_{M}|\Omega|^{-1}\varepsilon u\right)\varphi \, dx=0\quad \text{ for all }~\varphi\in\mathcal{C}_{c}^{\infty}(\Omega).
\end{align*}
In consequence, $u$ satisfies that $u+A|u|^{r-2}u-2\lambda_{M}|\Omega|^{-1}\varepsilon u=0$ a.e. in $\Omega.$
If $x_{1}\in\Omega$ is such that $u(x_{1})\neq 0$ then, $A|u(x_{1})|^{r-2}=2\lambda_{M}|\Omega|^{-1}\varepsilon-1.$ Therefore,
\begin{align}
|u|=K_{0}\chi_{V_{0}},\qquad V_0:=\{x\in \Omega\::\: u\neq 0\}\label{eq:ac9}
\end{align}
for some constant $K_{0}>0$. Since $u$ must satisfy that $G(u)=1$, it follows that 
\begin{align}
\label{eq:ac91}
K_{0}=\left(\frac{|\Omega|}{\varepsilon|V_{0}|}\right)^{1/2}, 
\end{align}
and in particular, $|u|_{r}^{r}=\frac{|\Omega|^{r/2}}{\varepsilon^{r/2}|V_{0}|^{(r-2)/2}}.$ Now, let us assume that $|V_{0}|<|\Omega|$. Given that $u$ is a minimizer, \eqref{eq:ac9}  combined with \eqref{eq:ac00} and Theorem~\ref{infimum} imply that
\begin{align*}
\Theta_{0}&=\frac{|u|_{2}^{2}}{2}+\frac{A|u|^{r}_{r}}{r}=|V_{0}|\left(\frac{1}{2\varepsilon}\frac{|\Omega|}{|V_{0}|}+\frac{A}{r\varepsilon^{r/2}}\frac{|\Omega|^{r/2}}{|V_{0}|^{r/2}}\right)\\
&>|V_{0}|\left(\frac{1}{2\varepsilon}\frac{|\Omega|}{|V_{0}|}+\frac{A}{r\varepsilon^{r/2}}\frac{|\Omega|}{|V_{0}|}\right)=\frac{|\Omega|}{2\varepsilon}+\frac{A}{r\varepsilon^{r/2}}|\Omega|=\Theta_{0},
\end{align*}
a contradiction.  Therefore, $|V_{0}|=|\Omega|$. This implies that $|\Omega\setminus V_{0}|=0$, which leads us to conclude that $\chi_{V_{0}}=\chi_{\Omega}$ a.e. in $\Omega$, and by \eqref{eq:ac91} that $K_{0}=\varepsilon^{-1/2}$. The result now follows from \eqref{eq:ac9}.
\end{proof}

Recall that $\lambda_{1,s}=\lambda_{1,s}(\Omega)>0$ denotes the first Dirichlet eigenvalue of the fractional Laplacian $(-\Delta)^s$ in a domain $\Omega$ (see \eqref{fef}).
\begin{proposition}
\label{infbound}
Let $\eps>0,$ $A>0,$ $r>2,$ and $\eta>\lambda_{1,s}(\Omega)$. There is a positive weak solution $u\in\mathcal{H}_{0}^{s}(\Omega)\cap L^{r}(\Omega)$ of the equation $(-\Delta)^{s}u+Au^{r-1}=\eta u$ in $\Omega,$ that is,
\begin{align}
\label{eq:acweak}
\mathcal{E}_{s}(u,\phi)+A\int_{\Omega}u^{r-1}\phi dx-\eta\int_{\Omega}u\phi dx=0 \qquad \text{ for all }\phi\in\mathcal{C}_{c}^{\infty}(\Omega).
\end{align}
Moreover, $u\leq\left(\frac{\eta}{A}\right)^{\tfrac{1}{r-2}}$ a.e. in $\R^N$.
\end{proposition}
\begin{proof}
The existence of $u$ follows by global minimization and standard arguments (see, for example, \cite[Corollary 6.3]{BFMST18}). Let $\eta_{0}:=(\tfrac{\eta}{A})^\frac{1}{r-2}$ and $\phi:=(\eta_{0}-u)_-=-\min\{0,\eta_{0}-u\}\geq 0;$ then,
\begin{align*}
    u(\eta_{0}^{r-2}-u^{r-2})\phi
    =u(\eta_{0}^{r-2}-u^{r-2})\frac{\eta_{0}-u}{\eta_{0}-u}\phi
    =-u\phi^2\frac{\eta_{0}^{r-2}-u^{r-2}}{\eta_{0}-u}\leq 0,
\end{align*}
since $(\eta_{0}^{r-2}-u^{r-2})/(\eta_{0}-u)>0$. Moreover, $u(x)-\eta_{0}
=-(\eta_{0}-u(x))
=-(\eta_{0}-u(x))_++\phi(x),$ thus $u(x)-u(y)
=(u(x)-\eta_{0})-(u(y)-\eta_{0})
=(\eta_{0}-u(y))_+-(\eta_{0}-u(x))_+
+\phi(x)-\phi(y),$
and
\begin{align*}
(u(x)-u(y))&(\phi(x)-\phi(y))
=(\phi(x)-\phi(y))^2
+[(\eta_{0}-u(y))_+-(\eta_{0}-u(x))_+](\phi(x)-\phi(y))\\
&=(\phi(x)-\phi(y))^2
+(\eta_{0}-u(y))_+\phi(x)
+(\eta_{0}-u(x))_+\phi(y)
\geq (\phi(x)-\phi(y))^2;
\end{align*}
but then, by \eqref{eq:acweak}, $0={\mathcal E}_s(u,\phi)+A\int_\Omega u(x)(u^{r-2}(x)-\eta_{0}^{r-2})\phi(x)\, dx
\geq {\mathcal E}_s(\phi,\phi)\geq 0,$ which implies that $\phi\equiv 0$ and $u\leq\eta_{0}$ in $\Omega$. 
\end{proof}

\begin{lemma}
\label{nonoptimal}
Let $v_{n}$ be as in \eqref{vn}. Then, the sequence $(v_{n})_{n\in\N}$ is bounded in $L^{\infty}(\Omega)$.
\end{lemma}
\begin{proof}
Since $v_{n}$ is a minimizer of $L_{n}$ (given in \eqref{Lndef}) under the restriction $G_{n}(u):=|\Omega|^{-1}\varepsilon|u|_{2}^{2}=1$, the Lagrange's multiplier theorem implies the existence of a real number $\lambda_{n}$ such that $v_{n}$ is a weak solution of $(-\Delta)^{s_n}v_n+Av_n^{r-1}=2\lambda_{n}|\Omega|^{-1}\varepsilon u$ in $\Omega.$ Moreover,
\begin{align}\label{ln}
\lambda_{n}=\frac{\|v_{n}\|_{s_{n}}^{2}+A|v_{n}|_{r}^{r}}{2}=\Theta_{n}+\left(\frac{r-2}{2r}\right)A|v_{n}|_{r}^{r},
\end{align}
where $\Theta_n$ is given in \eqref{thnn}. By Theorem~\ref{acaprox} it follows that $\lambda_{n}$ is bounded and, by Proposition~\ref{infbound}, $v_{n}\leq\left((2\lambda_{n}|\Omega|^{-1}\varepsilon)/A\right)^{\tfrac{1}{r-2}}$, which yields the result.
\end{proof}

\begin{theorem}
\label{acconvergence}
Let $v_{n}$ be as in \eqref{vn}. Then $v_{n}\rightarrow\varepsilon^{-1/2}$ in $L^{p}(\Omega)$ as $n\to\infty$ for every $1\leq p<\infty$.
\end{theorem}
\begin{proof}
By Theorems \ref{infimum}, \ref{acaprox}, and the fact that $v_n\in\Sigma_n$,
\begin{align}
\frac{A}{r}|v_{n}^{2}|_{r/2}^{r/2}
=
\frac{A}{r}|v_{n}|_{r}^{r}
=\Theta_0-\frac{|v_{n}|_{2}^{2}}{2}+o(1)
=\Theta_0-\frac{|\Omega|}{2\varepsilon}+o(1)
=\frac{A}{r}\frac{|\Omega|}{\varepsilon^{r/2}}+o(1)\label{eq:ac10}
\end{align}
as $n\rightarrow\infty,$ which implies that the sequence $(w_{n})_{n\in\N}:=(v_{n}^{2})_{n\in\N}$ is bounded in $L^{r/2}(\Omega)$. Then, there is $w^{\ast}\in L^{r/2}(\Omega)$ such that, up to a subsequence, 
\begin{align}
w_{n}\rightharpoonup w^{\ast}\quad\mbox{in}~L^{r/2}(\Omega)\quad\text{ as }n\to\infty.
\label{eq:ac11}
\end{align}
In consequence, $|w^{\ast}|_{r/2}^{r/2}\leq\liminf\limits_{n\rightarrow\infty}|w_{n}|_{r/2}^{r/2}=\liminf\limits_{n\rightarrow\infty}|v_{n}|_{r}^{r}.$ Then, by Theorem~\ref{acaprox},
\begin{align}
\frac{|\Omega|}{2\varepsilon}+\frac{A}{r}|w^{\ast}|_{r/2}^{r/2}\leq \frac{|\Omega|}{2\varepsilon}+\liminf\limits_{n\rightarrow\infty}\left(\frac{A}{r}|v_{n}|_{r}^{r}\right)=\Theta_{0}.\label{eq:ac12}
\end{align}
By \eqref{eq:ac11}, for every open set $\mathcal{O}\subset\Omega$,
\begin{align}
0\leq \int_{\mathcal{O}}v_{n}^{2}\, dx=\int_{\Omega}v_{n}^{2}\chi_{\mathcal{O}}\, dx\rightarrow\int_{\Omega}w^{\ast}\chi_{\mathcal{O}}\, dx=\int_{\mathcal{O}}w^{\ast}\, dx.\label{eq:ac13}
\end{align}
Hence, $\int_{\mathcal{O}}w^{\ast}\, dx\geq 0$ for every open set $\mathcal{O}\subset\Omega$ and thus, Lebesgue's differentiation theorem yields that $w^{\ast}\geq 0$ a.e. in $\Omega$. Moreover, taking $\mathcal{O}=\Omega$ in \eqref{eq:ac13}, $|\Omega|\varepsilon^{-1}=\int_{\Omega}v_{n}^{2}\, dx\rightarrow\int_{\Omega}w^{\ast}\, dx.$  Therefore, $\int_{\Omega}|w^{\ast}|^{r/2}\, dx=|\sqrt{w^{\ast}}|_{r}^{r}$ and $\int_{\Omega}w^{\ast}\, dx=|\sqrt{w^{\ast}}|^{2}_{2}=|\Omega|\varepsilon^{-1}.$
Then, \eqref{eq:ac12} yields the inequality $
\frac{1}{2}|\sqrt{w^{\ast}}|_{2}^{2}+\frac{A}{r}|\sqrt{w^{\ast}}|_{r}^{r}\leq\Theta_{0},$ which implies that $\sqrt{w^{\ast}}\in L^{r}(\Omega)$  is a minimizer of the functional $J(u)$ with the restriction $G(u)-1=0$. Consequently, Theorem~\ref{acmin} yields that $\sqrt{w^{\ast}}=\varepsilon^{-1/2}\chi_{\Omega}$. From \eqref{eq:ac10} and \eqref{eq:ac11},
\begin{align}
v_{n}^{2}\rightharpoonup\frac{1}{\varepsilon}\quad\mbox{in}~L^{r/2}(\Omega)\qquad \text{ as }n\to\infty.\label{eq:alle1}
\end{align}
Since \eqref{eq:ac10} means that $|v_{n}^{2}|_{r/2}^{r/2}=\tfrac{\Omega}{\varepsilon^{r/2}}+o(1)$ as $n\rightarrow\infty$, this result together with \eqref{eq:alle1} implies that $v_{n}^{2}\rightarrow\varepsilon^{-1}$ in $L^{r/2}(\Omega)$ as $n\rightarrow\infty$. Finally, since $(v_{n})_{n}$ is bounded in $L^{\infty}(\Omega)$ and, up to a subsequence, $v_{n}\rightarrow\varepsilon^{-1/2}$ a.e. in $\Omega$, from the dominated convergence theorem it follows that $v_{n}\rightarrow\varepsilon^{-1/2}$ in $L^{p}(\Omega)$ for every $1\leq p<\infty$, as desired.  Since the limit is independent of the chosen subsequence, the convergence holds for the whole sequence, as claimed. 
\end{proof}
Finally, as a consequence of this last result, we can show that the bound obtained during the proof of Lemma~\ref{nonoptimal} can be improved.

\begin{corollary}
\label{bestacbound}
Let $(v_{n})_{n\in\N}$ be as in \eqref{vn}, then
\begin{align*}
0\leq v_{n}\leq\left(\frac{1}{A}+\varepsilon^{\tfrac{2-r}{2}}\right)^{\tfrac{1}{r-2}}+o(1)\quad\mbox{as}~n\rightarrow\infty.
\end{align*}
\end{corollary}
\begin{proof}
By Proposition~\ref{infbound}, we have that $v_{n}\leq A^{\tfrac{1}{2-r}}\left(2\lambda_{n}|\Omega|^{-1}\varepsilon\right)^{\tfrac{1}{r-2}}$. Using \eqref{ln},
\begin{align*}
v_{n}\leq A^{\tfrac{1}{2-r}}\left\lbrace 2|\Omega|^{-1}\varepsilon\left(\Theta_{n}+\left(\frac{r-2}{2r}\right)A|v_{n}|_{r}^{r}\right)\right\rbrace^{\tfrac{1}{r-2}}.
\end{align*}
Since, by Theorem~\ref{acconvergence}, $|v_{n}|_r^r\to \varepsilon^{-r/2}|\Omega|$ , we have, by Theorems~\ref{infimum} and \ref{acaprox}, that
\begin{align*}
v_{n}
\leq A^{\tfrac{1}{2-r}}\left\lbrace 2|\Omega|^{-1}\varepsilon\left(\frac{|\Omega|}{2\varepsilon}+\frac{A|\Omega|}{2\varepsilon^{r/2}}+o(1)\right)\right\rbrace^{\tfrac{1}{r-2}}
= A^{\tfrac{1}{2-r}}\left(1+A\varepsilon^{\tfrac{2-r}{2}}\right)^{\tfrac{1}{r-2}}+o(1)\quad\text{as $n\rightarrow\infty$.}
\end{align*}

\end{proof}

The following is an easy calculation that will be useful for our next result. 
\begin{lemma}\label{Flem}
For $M,r>2$, $a\in[0,M]$, $b\geq 0$, $a\neq b$, let $F(a,b):=\frac{a^{r-2}-b^{r-2}}{a-b}$. There are $C=C(r,M)>0$ and $\alpha=\alpha(r)\geq 0$ such that $F(a,b)\geq C b^{\alpha}$.
\end{lemma}
\begin{proof}
If $r-2>1$ and $z:=\tfrac{a}{b}$, then $\frac{F(a,b)}{a^{r-3}+b^{r-3}}=\frac{z^{r-2}-1}{(z-1)(z^{r-3}+1)}.$ Since $\lim_{z\rightarrow 1}\frac{z^{r-2}-1}{(z-1)(z^{r-3}+1)}=\frac{r-2}{2},$ we can find $C=C(r)>0$ such that $F(a,b)\geq C(r)(a^{r-3}+b^{r-3})$ for all $n\in\mathbb{N}$.

If $0<r-2<1$, then the function $f(y)=y^{r-2}$ is concave, which implies that $F(a,b)\geq F(M,b)$ for all $a<M$ and $b\in\mathbb{R},$ 
where $\lim\limits_{b\rightarrow M}F(M,b)=(r-2)M^{r-3}$. Therefore, there is $C_{0}=C_{0}(r,M)>0$ such that $F(a,b)\geq C_{0}>0$.
\end{proof}

We are ready to show the main result in this section.

\begin{theorem}
\label{ac1convergence}
Let $\eps>0,$ $A>0$, $r>2$, $\eta_0:=1+A\varepsilon^{\frac{2-r}{2}}$, and let $(s_n)_{n\in\N}\subset (0,1)$ be such that $\lim_{n\to\infty}s_n=0$.  For $n$ sufficiently large, the problem
\begin{align}
\label{eq:ac14}
(-\Delta)^{s_n}u_{n}+Au_{n}^{r-1}-\eta_0 u_n=0\quad \text{ in }\Omega, \qquad u_n=0\quad \text{ on }\R^N\backslash \Omega,
\end{align}
 has a unique positive solution $u_n\in \cH^{s_n}_0(\Omega)\cap L^{r}(\Omega)$. Moreover,
 \begin{align*}
\text{ $u_{{n}}\to \varepsilon^{-1/2}$ \qquad in $L^{p}(\Omega)$ as $n\rightarrow\infty$  for every $1\leq p<\infty$.}
 \end{align*}
\end{theorem}
\begin{proof}
Since $\lim_{n\to\infty}s_n=0$, by \eqref{asymp}, there is $n_0\in\mathbb N$ so that $\eta_0:=1+A\varepsilon^{\frac{2-r}{2}}>\lambda_{1,s_n}$ for all $n\geq n_0$.  Then, the existence and uniqueness of a positive solution $u_n\in \cH^{s_n}_0(\Omega)\cap L^{r}(\Omega)$ of \eqref{eq:ac14} follows by arguing as in \cite[Corollary~6.3]{BFMST18}.

Let $v_n$ and $\Theta_n$ be as \eqref{vn}, and $\lambda_n$ be as in \eqref{ln}.  In particular,
\begin{align}
(-\Delta)^{s_{n}}v_{n}+Av_{n}^{r-1}-\eta_n v_{n}=0\qquad \text{ in }\Omega,\qquad \eta_n:=2|\Omega|^{-1}\varepsilon\lambda_{n}.\label{eq:ac15}
\end{align}
By \eqref{ln} and Theorems \ref{acaprox} and \ref{acconvergence}, we have that $\eta_n\to \eta_0$ as $n\to\infty$. 
Let $w_{n}:=u_{n}-v_{n}$, then
\begin{align*}
(-\Delta)^{s_{n}}w_{n}+\left(Au_{n}^{r-2}-\eta_{0}\right)w_{n}=\left(\eta_{0}-\eta_{n}-A(u_{n}^{r-2}-v_{n}^{r-2})\right)v_{n}\quad \text{ in }\Omega,
\end{align*}
Define $F(a,b):=\frac{a^{r-2}-b^{r-2}}{a-b},$ and notice that $F>0$ for $a\neq b$, $a,b\geq 0$. Then,
\begin{align}
\|w_n\|_{s_{n}}^2+\int_{\Omega}\left(Au_{n}^{r-2}-\eta_{0}\right)w_{n}^{2}\, dx&=(\eta_{0}-\eta_{n})\int_{\Omega}v_{n}w_{n}\, dx-A\int_{\Omega}F(u_{n},v_{n})w_{n}^{2}v_{n}\, dx\nonumber\\
&\leq (\eta_{0}-\eta_{n})\int_{\Omega}v_{n}w_{n}\, dx=o(1),\label{eq:ac17}
\end{align}
because $\eta_{n}\rightarrow\eta_{0}$ as $n\rightarrow\infty$ and because $w_n,v_n\in L^\infty(\Omega)$, by Proposition~\ref{infbound} and Corollary \ref{bestacbound}. 

Now we argue as in \cite[Proposition 6.2]{BFS16}. By using standard arguments, the problem 
\begin{align}\label{ss}
\mu_{n}=\inf_{v\in\mathcal{H}_{0}^{s_{n}}(\Omega)\setminus\left\lbrace 0\right\rbrace}\frac{\|v\|_{s_{n}}^{2}+\int_{\Omega}\left(Au_{n}^{r-2}-\gamma_{0}\right)v^{2}\, dx}{|v|_{2}^{2}},
\end{align}
has a nontrivial non-negative solution $z_{n}\in\mathcal{H}_{0}^{s_{n}}(\Omega)$ for each $n\in\mathbb{N}$. In particular, $z_{n}$ is a weak solution of $
(-\Delta)^{s_{n}}z_{n}+\left(Au_{n}^{r-2}-\gamma_{0}\right)z_{n}=\mu_{n}z_{n}$ in $\Omega.$ Testing with $u_{n}$ and integrating by parts,
\begin{align}\label{sp}
0=\int_{\Omega}\left((-\Delta)^{s_{n}}u_{n}+\left(Au_{n}^{r-2}-\eta_{0}\right)u_{n}\right)z_{n}\, dx=\mu_{n}\int_{\Omega}z_{n}u_{n}\, dx,
\end{align}
by \eqref{eq:ac14}. Let us show that $\mu_{n}=0$.  By Proposition~\ref{infbound}, $u_n\leq (\eta_0/A)^\frac{1}{r-2}$, and then $(-\Delta)^{s_{n}}u_{n}=\left(\eta_{0}-Au_{n}^{r-2}\right)u_{n}\geq 0$ in $\Omega$; by \eqref{eq:ac00}, we can apply the strong maximum principle (see, for example, \cite{JW19}) to conclude that $u_{n}>0$ in $\Omega$. Since $z_{n}\geq 0$ and $z_n\neq 0$, \eqref{sp} implies that $\mu_{n}=0$. Then, by \eqref{eq:ac17} and the definition of $\mu_n$, 
\begin{align*}
0=\mu|w_n|_2^2\leq \|w_n\|_{s_n}^2+\int_\Omega (Au_n^{r-2}-\eta_0)w_n^2\, dx =o(1)-A\int_\Omega F_n(u_n,v_n)w_n^2v_n\, dx\leq o(1)
\end{align*}
as $n\to\infty.$ In particular, $\lim\limits_{n\rightarrow\infty}A\int_{\Omega}F(u_{n},v_{n})w_{n}^{2}v_{n}\, dx=0$. Since Proposition~\ref{infbound} guarantees the existence of a constant $M>0$ such that $u_{n}\leq M$ for all $n\in\mathbb{N}$, we have, by Lemma~\ref{Flem}, that there are $C_{1}=C_{1}(r,M)>0$ and $\alpha=\alpha(r)\geq 0$ such that $F(u_n,v_n)\geq C_{1}v_n^{\alpha}$. As a consequence, 
\begin{align}
\label{eq:ac18}
0=\lim_{n\rightarrow\infty}A\int_{\Omega}F(u_{n},v_{n})w_{n}^{2}v_{n}\, dx\geq C_1\lim_{n\rightarrow\infty}\int_{\Omega}v_{n}^{\alpha+1}w_{n}^{2}\, dx\geq 0,
\end{align}
that is, $\lim\limits_{n\rightarrow\infty}\int_{\Omega}v_{n}^{\alpha+1}w_{n}^{2}\, dx=0$. Furthermore, by Theorem~\ref{acconvergence} and dominated convergence, we have that  $\lim_{n\to\infty}\int_{\Omega}|1-\varepsilon^{\tfrac{\alpha+1}{2}} v_{n}^{\alpha+1}|\, dx=0$. By Proposition~\ref{infbound} and Corollary \ref{bestacbound}, there is $C>0$ such that $|w_n|^2_\infty<C$ and then
\begin{align*}
0\leq \int_{\Omega}w_{n}^{2}\, dx
\leq \int_{\Omega}(1-\varepsilon^{\tfrac{\alpha+1}{2}} v_{n}^{\alpha+1})w_{n}^{2}\, dx+\varepsilon^{\tfrac{\alpha+1}{2}}\int_{\Omega}v_{n}^{\alpha+1}w_{n}^{2}\, dx=o(1)
\end{align*}
as $n\to\infty$, \emph{i.e.}, $\lim\limits_{n\rightarrow\infty}\int_{\Omega}w_{n}^{2}\, dx=0.$ Finally,
\begin{align*}
\int_{\Omega}|u_{n}-\varepsilon^{-1/2}|^{2}\, dx\leq\int_{\Omega}\left(|w_{n}|+|v_{n}-\varepsilon^{-1/2}|\right)^{2}\, dx\leq 4\int_{\Omega}w_{n}^{2}+|v_{n}-\varepsilon^{-1/2}|^{2} dx\rightarrow 0\quad\mbox{as}~n\rightarrow\infty,
\end{align*}
which proves the result for $p=2$. The general case, $1\leq p<\infty$, now follows from the dominated convergence theorem since, by Proposition~\ref{infbound}, $(u_{n})_{n}$ is bounded in $L^{\infty}(\Omega)$.
\end{proof}

\begin{proof}[Proof of Theorem~\ref{thm:f}]
The proof follows directly from Theorem~\ref{ac1convergence} using $r=p+1$, $A=1,$ and $\eps=(k-1)^\frac{2}{2-r}=(k-1)^\frac{2}{1-p}$.
\end{proof}

\small{
\subsection*{Acknowledgments}
We thank the anonymous referees for their helpful comments and suggestions. We also thank Héctor Chang, Mónica Clapp, and Sven Jarohs for helpful comments and suggestions. F. Angeles is  supported by a CONACyT postdoctoral grant A1-S-10457 (Mexico). A. Saldaña is supported  by CONACYT grant A1-S-10457 (Mexico) and by UNAM-DGAPA-PAPIIT grants IA101721 and IA100923 (Mexico).

\subsection*{Conflicts of interests}
 The authors declare no conflicts of interests. 
}

\end{document}